\documentclass[11pt,a4paper]{article}

\usepackage{amsmath,amsfonts,amssymb,latexsym,graphics,epsfig,url}
\usepackage{color}
\usepackage{amsthm}
\usepackage[english]{babel}
\usepackage{graphicx}
\usepackage{soul}
\usepackage[utf8]{inputenc}

\newcommand\blfootnote[1]{%
  \begingroup
  \renewcommand\thefootnote{}\footnote{#1}%
  \addtocounter{footnote}{-1}%
  \endgroup
}

\newtheorem{theorem}{Theorem}[section]
\newtheorem{proposition}{Proposition}[section]
\newtheorem{lemma}[theorem]{Lemma}
\newtheorem{corollary}[theorem]{Corollary}

\newtheorem{conjecture}[theorem]{Conjecture}
\newtheorem{example}[theorem]{Example}

\DeclareMathOperator{\Ker}{Ker}
\DeclareMathOperator{\rank}{rank}
\DeclareMathOperator{\spec}{sp}

\def\v{\mbox{\boldmath $v$}}
\def\x{\mbox{\boldmath $x$}}

\def\w{\mbox{\boldmath $w$}}

\def\vec0{\mbox{\boldmath $0$}}

\def\A{\mbox{\boldmath $A$}}
\def\B{\mbox{\boldmath $B$}}
\def\C{\mbox{\boldmath $C$}}

\def\D{\mbox{\boldmath $D$}}

\def\I{\mbox{\boldmath $I$}}
\def\J{\mbox{\boldmath $J$}}

\def\L{\mbox{\boldmath $L$}}
\def\M{\mbox{\boldmath $M$}}
\def\O{\mbox{\boldmath $O$}}

\def\T{\mbox{\boldmath $T$}}
\def\U{\mbox{\boldmath $U$}}

\def\W{\mbox{\boldmath $W$}}

\def\I{\mbox{\boldmath $I$}}
\def\J{\mbox{\boldmath $J$}}

\def\1{\mbox{\boldmath $1$}}

\def\Re{\mathbb R}

\newcommand\restr[2]{\ensuremath{\left.#1\right|_{#2}}}

\begin{document}

\title{On the Laplacian spectra of token graphs
\thanks{This research of C. Dalf\'o and M. A. Fiol has been partially supported by 
AGAUR from the Catalan Government under project 2017SGR1087 and by MICINN from the Spanish Government under project PGC2018-095471-B-I00. The research of C. Dalf\'o has also been supported by MICINN from the Spanish Government under project MTM2017-83271-R. The research of C. Huemer was supported by PID2019-104129GB-I00/ AEI/ 10.13039/501100011033 and Gen. Cat. DGR 2017SGR1336. F. J. Zaragoza Martínez acknowledges the support of the National
Council of Science and Technology (Conacyt) and its National System of Researchers (SNI).}
}
\author{
	C. Dalf\'o$^a$, F. Duque$^b$, R. Fabila-Monroy$^c$, M. A. Fiol$^d$,\\ C. Huemer$^e$, A. L. Trujillo-Negrete$^f$, F. J. Zaragoza Mart\'inez$^g$\\
	\\
	{\small $^a$Dept. de Matem\`atica, Universitat de Lleida, Igualada (Barcelona), Catalonia}\\
	{\small {\tt cristina.dalfo@udl.cat}}\\
	{\small $^{b}$ Instituto de Matemáticas, Universidad de Antioquia, Medellín, Colombia}\\
	{\small {\tt rodrigo.duque@udea.edu.co}}\\
	{\small $^{c,f}$Departamento de Matemáticas, Cinvestav, Mexico City, Mexico}\\
	{\small {\tt ruyfabila@math.cinvestav.edu.mx}, {\tt ltrujillo@math.cinvestav.mx}}\\
	{\small $^{d,e}$Dept. de Matem\`atiques, Universitat Polit\`ecnica de Catalunya, Barcelona, Catalonia} \\
	{\small $^{d}$Barcelona Graduate School of Mathematics} \\
	{\small {\tt miguel.angel.fiol@upc.edu}, {\tt clemens.huemer@upc.edu}} \\
	{\small $^{g}$Departamento de Sistemas, Universidad Aut\'onoma Metropolitana Azcapotzalco,}\\
	{\small Mexico City, Mexico}\\
	{\small {\tt franz@azc.uam.mx}}\\
}

\date{}
\maketitle

\blfootnote{
\begin{minipage}[l]{0.3\textwidth} \includegraphics[trim=10cm 6cm 10cm 5cm,clip,scale=0.15]{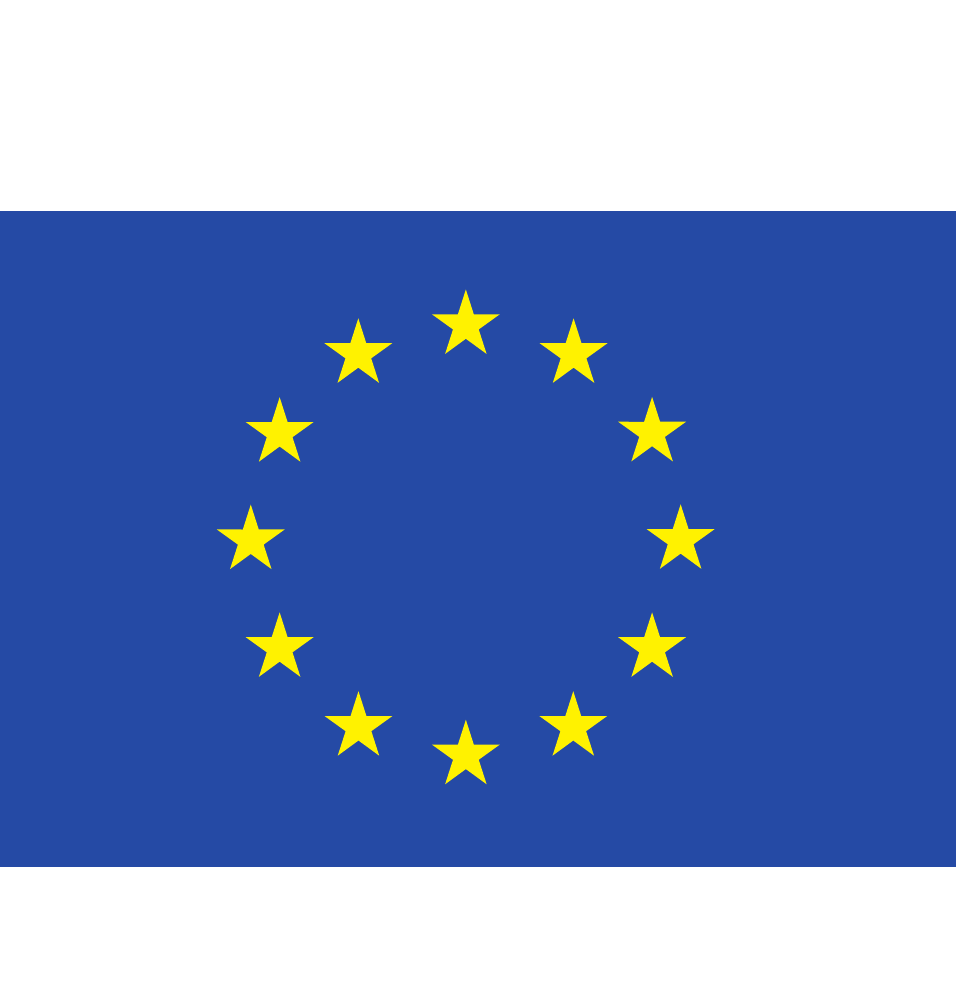} \end{minipage}  \hspace{-2cm} \begin{minipage}[l][1cm]{0.79\textwidth}
   This research has also received funding from the European Union's Horizon 2020 research and innovation programme under the Marie Sk\l{}odowska-Curie grant agreement No 734922.
  \end{minipage}}

\begin{abstract}
We study the Laplacian spectrum of token graphs, also called symmetric powers of graphs. The $k$-token graph $F_k(G)$ of a graph $G$ is the
graph whose vertices are the $k$-subsets of vertices from $G$, two of which being adjacent whenever their symmetric difference is a pair of adjacent vertices in $G$.
In this paper, we give a relationship between the Laplacian  spectra of any two token graphs of a given graph.
In particular, we show that, for any integers $h$ and $k$ such that $1\le h\le k\le \frac{n}{2}$, the Laplacian spectrum of $F_h(G)$ is contained in the Laplacian spectrum of $F_k(G)$.
We also show that the  double odd graphs and doubled Johnson graphs  can be obtained  as token graphs of the complete graph $K_n$ and the star $S_{n}=K_{1,n-1}$, respectively.
Besides, we obtain a relationship between the spectra of the $k$-token graph of $G$ and the $k$-token graph of its complement $\overline{G}$. This generalizes a well-known property for Laplacian eigenvalues of graphs to token graphs.
Finally, the  double odd graphs and doubled Johnson graphs provide two infinite families, together with some others, in which the algebraic connectivities of the original graph and its token graph coincide. Moreover, we conjecture that this is the case for any graph $G$ and its token graph.
\end{abstract}

\noindent{\em Keywords:} Token graph, Laplacian spectrum, Algebraic connectivity, Binomial matrix, Adjacency spectrum, Double odd graph, Doubled Johnson graph, Complement graph.

\noindent{\em MSC2010:} 05C15, 05C10, 05C50.

\section{Introduction}
\label{sec:-1}
Let $G$ be a simple graph with vertex set $V(G)=\{1,2,\ldots,n\}$ and edge set $E(G)$. For a given integer $k$ such that
$1\le k \le n$, the {\em $k$-token graph} $F_k(G)$ of $G$ is the graph whose vertex set $V (F_k(G))$ consists of the ${n \choose k}$
$k$-subsets of vertices of $G$, and two vertices $A$ and $B$
of $F_k(G)$ are adjacent whenever their symmetric difference $A \bigtriangleup B$ is a pair $\{a,b\}$ such that $a\in A$, $b\in B$, and $\{a,b\}\in E(G)$; see Figure \ref{fig1} for an example. 
This naming
comes from an observation in
Fabila-Monroy,  Flores-Pe\~{n}aloza,  Huemer,  Hurtado,  Urrutia, and  Wood \cite{ffhhuw12}, that vertices of $F_k(G)$ correspond to configurations
of $k$ indistinguishable tokens placed at distinct vertices of $G$, where
two configurations are adjacent whenever one configuration can be reached
from the other by moving one token along an edge from its current position
to an unoccupied vertex. Such graphs are also called {\em symmetric $k$-th power of a graph} in Audenaert, Godsil, Royle, and Rudolph \cite{agrr07}; and {\em $n$-tuple vertex graphs} in Alavi,  Lick, and Liu \cite{all02}. They have applications in physics; a
connection between symmetric powers of graphs and the exchange of Hamiltonian operators in
quantum mechanics is given in \cite{agrr07}. Our interest is in relation to the graph
isomorphism problem. It is well known that there are cospectral non-isomorphic
graphs, where often the spectrum of the adjacency matrix of a graph is used. For instance,
Rudolph \cite{r02} showed that there are cospectral non-isomorphic graphs that can
be distinguished by the adjacency spectra of their 2-token graphs, and he also gave an example for the Laplacian spectrum. Audenaert, Godsil, Royle, and by Rudolph \cite{agrr07}
proved that 2-token graphs of strongly regular graphs with the same parameters
are cospectral, and also derived bounds on the (adjacency and Laplacian) eigenvalues of $F_2(G)$ for general graphs. The adjacency spectrum of some $k$-token graphs was also studied by Barghi and
Ponomarenko \cite{bp09}, and Alzaga, Iglesias, and Pignol \cite{aip10}, who proved that for each value of $k$
there are infinitely many pairs of non-isomorphic graphs with cospectral $k$-token
graphs.

In 2012, Fabila-Monroy, Flores, Huemer, Hurtado, Urrutia, and Wood  \cite{ffhhuw12} conjectured that  a given graph $G$ is determined, up to isomorphism, by its $k$-token graph $F_k(G)$ for some fixed $k$. This has been proved true for different classes of graphs and $k=2$  (called {\em double vertex graphs}), such as trees  (see Alavi, Behzad, Erd\H{o}s, Lick \cite{abel91}), regular graphs without $4$-cycles (see Jacob, Goddard, Laskar\cite{jgl07}), 
and cubic graphs (also \cite{jgl07}).

\begin{figure}[t]
\begin{center}
\includegraphics[width=10cm]{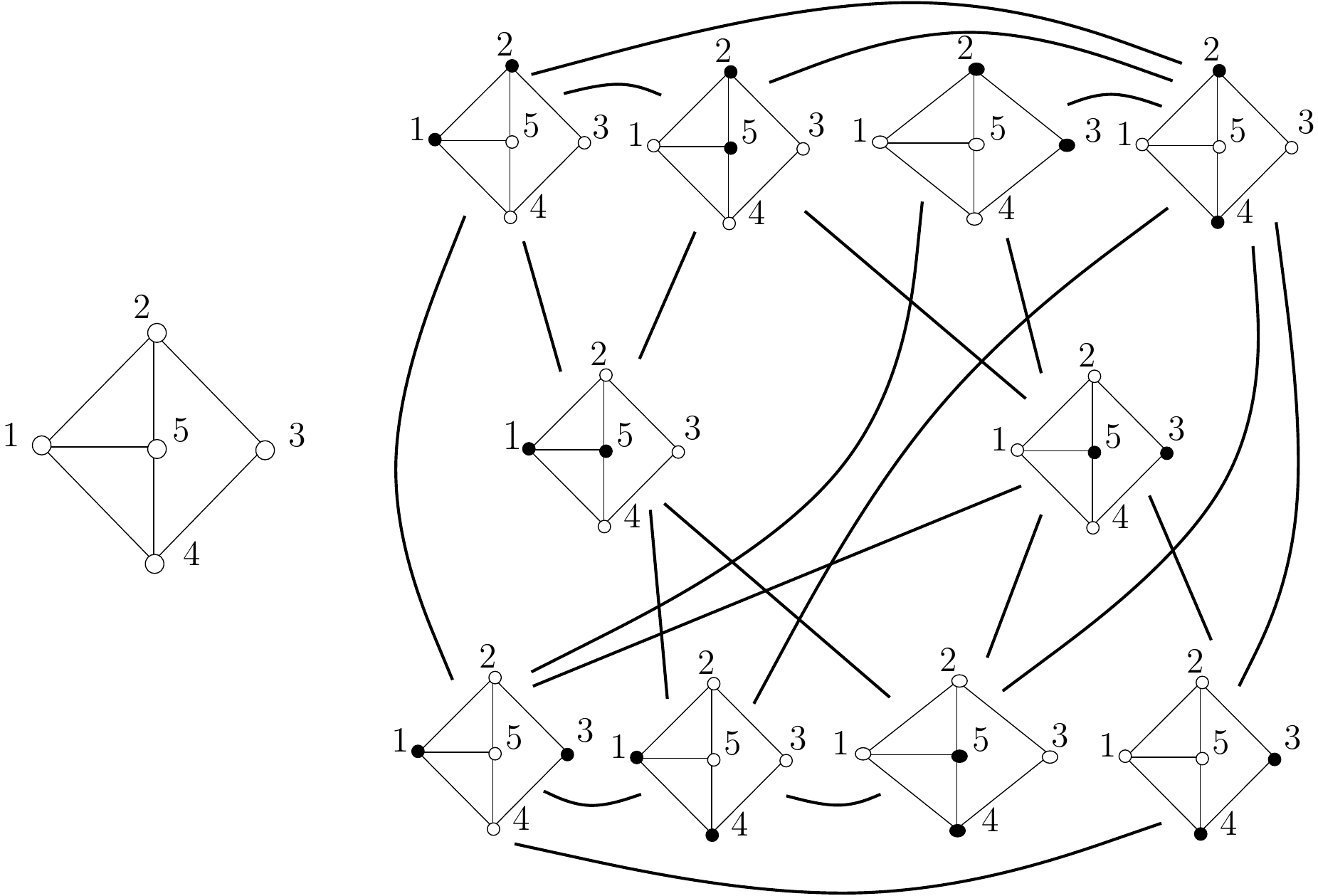}
\caption{A graph $G$ (left) and its $2$-token graph $F_2(G)$ (right). The Laplacian spectrum of $G$ is $\{0,2,3,4,5\}$.  The Laplacian spectrum of $F_2(G)$ is $\{0,2,3^2,4,5^3,7,8\}$. }
\label{fig1}
\end{center}
\end{figure}

In this paper, we focus on the Laplacian spectrum of $F_k(G)$ for any value of $k$.
Recall that the Laplacian matrix $\L(G)$ of a graph $G$ is $\L(G) = \D(G) - \A(G)$, where $\A(G)$ is the adjacency
matrix of $G$, and $\D(G)$ is the diagonal matrix whose diagonal entries are the vertex
degrees of $G$. For a $d$-regular graph $G$, each eigenvalue $\lambda$ of $\L(G)$ corresponds to an eigenvalue $\mu$ of $\A(G)$ via the relation $\lambda=d-\mu$.
In \cite{cflr17}, Carballosa, Fabila-Monroy, Lea\~nos, and  Rivera proved
that, for $1 < k < n - 1$, the $k$-token graph $F_k(G)$ is regular only if $G$ is the complete graph $K_n$
or its complement, or if $k=n/2$ and $G$ is the star graph $K_{1,n-1}$ or its complement.
Then, for most graphs, we cannot directly infer the Laplacian spectrum of $F_k(G)$ from the adjacency spectrum of $F_k(G)$. In fact, when considering the
adjacency spectrum, we find graphs $G$ whose spectrum is not contained
in the spectrum of $F_k(G)$; see Rudolph \cite{r02}. Surprisingly, for the Laplacian spectrum, this holds and it is our first result.

This paper is organized as follows. In the following section, we recall some basic notations and results. In Section \ref{sec:-3}, we prove that the Laplacian spectrum of a graph $G$ is contained in the Laplacian spectrum of its $k$-token graph $F_k(G)$. In Subsection \ref{sec:-3.1}, with the use of a new $(n;k)$-binomial matrix, we give the relationship between the Laplacian spectrum of a graph $G$ and that of its $k$-token graph.
In Section \ref{sec:-4}, we prove that an eigenvalue of a $k$-token graph is also an eigenvalue of the $(k+1)$-token graph for $1\leq k< n/2$. Besides, we define another matrix, called $(n;h,k)$-binomial matrix. With the use of this matrix, it is shown that, for any integers $h$ and $k$ such that $1\le h\le  k\le \frac{n}{2}$, the Laplacian spectrum of $F_h(G)$ is contained in the spectrum of $F_k(G)$.
In Section \ref{sec:-5}, we show that the  double odd graphs and doubled Johnson graphs  can be obtained  as token graphs of the complete graph $K_n$ and the star $S_{n}=K_{1,n-1}$, respectively. In Section \ref{sec:-6}, we obtain a relationship between the Laplacian spectra of the $k$-token graph of $G$ and the $k$-token graph of its complement $\overline{G}$. This generalize a well-known property for Laplacian eigenvalues of graphs to token graphs.
Finally, in the last section, the  double odd graphs and doubled Johnson graphs provide two infinite families, together with some others, in which the algebraic connectivities of the original graph and its token graph coincide. Moreover, we conjecture that this is the case for any graph $G$ and its token graph.

\section{Preliminaries}
\label{sec:-2}

Let us first introduce some notation used throughout the paper.
Given a graph $G=(V,E)$, we indicate with $a\sim b$ that $a$ and
$b$ are adjacent in $G$, and with $N_G(a)$  the set of vertices adjacent to vertex $a\in V$.
Similarly, if $A\subset V$, $N_G(A)$ denotes the (open) neighborhood of $A$, that is, the vertices not in $A$ that are adjacent to vertices of $A$.
\\
For disjoint $X,Y\subset V$, let $E_G(X,Y)$ be the set of edges of $G$ with one end in $X$ and the other end in $Y$. When $Y=\overline{X}$, the complement of $X$, $E_G(X,\overline{X})$ is simply the edge cut defined by $X$, usually denoted by $\partial_G(X)$. If, furthermore, $X=\{a\}$, there is often a slight abuse of notation to use $\partial_G(a)$, and then $\deg_G(a)=|\partial_G(a)|$. Note that, in general, $\deg_G(a)=|N_G(a)|$ only if $G$ is simple (which is our case).
For $B \subset  N_{G}(A)$ and $b \in B$, we denote by $\deg_G(b, A)=|N_G(b)\cap A|$ the number of neighbors of $b$ in $A$.
\\
As usual, the transpose of a matrix $\M$ is denoted by $\M^\top$, the
identity matrix by $\I$, the all-$1$ vector $(1,..., 1)^{\top}$ by $\1$, the all-$1$ (universal) matrix  by $\J$, and the all-$0$ vector and all-$0$ matrix by $\vec0$
and $\O$, respectively.
\\
Let $[n]:=\{1,\ldots,n\}$. Let ${[n]\choose k}$ denote the set of $k$-subsets of $[n]$, the set of vertices of the $k$-token graph.
\\
\\
For our purpose, it is convenient to denote by $W$ the set of all column vectors $\v$ such that  $\v^{\top }\1 = 0$. Any square matrix $\M$ with all zero row sums has an eigenvalue $0$ with corresponding eigenvector $\1$.
When $\M=\L(G)$, the Laplacian matrix of a graph $G$, the matrix is positive semidefinite, and its smallest eigenvalue is known as the {\em algebraic connectivity} of  $G$, here denoted by $\alpha(G)$.
\\
Given a graph $G=(V,E)$ of order $n$,
we say that a vector $\v\in \mathbb{R}^n$ is an \textit{embedding} of $G$ if $\v\in W$.
Note that if $\v$ is a $\lambda$-eigenvector of $G$, with $\lambda>0$, then it is an embedding of $G$.
For a graph $G$ with Laplacian matrix $\L(G)$, and an embedding $\v$ of $G$, let 
$$
\lambda(\v):=\frac{\v^{\top}\L(G)\v}{\v^{\top}\v}.
 $$
The value of $\lambda(\v)$ is known as
 the {\em Rayleigh quotient}.
If $\v$ is an eigenvector of $G$, then its corresponding eigenvalue is $\lambda(\v)$.
We will use the following well-known result:
For an embedding $\v$ of $G$, we have
\[
\lambda(\v)=\frac{\sum\limits_{(i,j)\in E}(\v(i)-\v(j))^2}{\sum\limits_{i\in V}\v^2(i)},\qquad \mbox{and}\qquad
\alpha(G)=\min\{\lambda(\v):\v\in W\},
\]
with the minimum occurring only when $\v$ is an $\alpha(G)$-eigenvector of $G$, and where
$\v(i)$ denotes the entry of $\v$ corresponding to the vertex $i\in V(G)$.
%

\section{The Laplacian spectra of token graphs}
\label{sec:-3}

Our first theorem deals with the Laplacian spectrum of a graph $G$ and its $k$-token graph $F_k(G)$.
\begin{theorem}
\label{th:(1,k)}
Let $G$ be a graph  and $F_k(G)$  its $k$-token graph. Then, the
Laplacian spectrum of $G$ is contained in the Laplacian spectrum of $F_k(G)$.
\end{theorem}
\begin{proof}
Let $G$ have order $n=|V(G)|$. Let $A=\{a_1,\ldots,a_k\}$ be a vertex of $F_k(G)$.
Let $\L:=\L(G)$ and $\L_k:=\L(F_k)$ denote the Laplacian matrices of $G$ and $F_k(G)$, respectively. Let $\v_G$ be an eigenvector of $\L$ with eigenvalue $\lambda$.
Using $\v_G$, we construct an eigenvector $\v_{F}(\neq\vec0)$ of $\L_k$, with ${n\choose k}$ entries and the same eigenvalue $\lambda$, as follows:
\begin{equation}
\label{vF(A)}
\v_{F}(A)
=\sum_{a\in A}\v_G(a)=\sum_{i=1}^k \v_G(a_i),\quad
\mbox{ for  $A\in V(F_k)$}.
\end{equation}
This means that vertex $A$ of $F_k(G)$ has tokens placed on vertices $a_1,\ldots,a_k$ of
$G$, and the corresponding entry of vertex $A$ of $F_k(G)$ in $\v_F$ is the sum of the token positions of $A$ in $G$, when embedded on the real line as eigenvector $\v_G$.

Let us now verify that $\v_F$ is indeed an eigenvector of $\L_k$ with associated eigenvalue $\lambda$. In the second part of the proof, we will show that $\v_F\neq \vec0$.
We can assume $\lambda > 0$.
We first observe that
$$
\sum_{A\in V(F_k)}\v_F(A)={n-1\choose k-1}\sum_{a\in V(G)} \v_G(a)=0,
$$
because each token placed at a vertex of $G$ appears in
${n-1\choose k-1}$ vertices of $F_k(G)$. Also, note that $\sum_{a\in V(G)} \v_G(a)=0$ is satisfied since $\v_G$ is an eigenvector orthogonal to the all-$1$ vector $\1$. Then, $\v_F$ is also orthogonal to $\1$.

We show that
$
\L_k  \v_F = \lambda \v_F
$
holds. More precisely, we show that each equation of this linear system holds.
Let us look at one row of matrix $\L_k$. Let this row correspond to vertex $A$ of $F_k$. Let $\deg_{F_k} (A)$ be the vertex degree of $A$ in $F_k$. Let $N_{F}(A)$ be the set of neighbors of $A$ in $F_k$. Then, we need to show that
\begin{equation}
\label{eq:1}
\deg_{F_k}(A)\v_{F}(A)\ -\sum_{B\in N_{F}(A)}\v_{F}(B)=\lambda \v_{F}(A).
\end{equation}
From \eqref{vF(A)}, $\v_{F}(A)=\sum_{a\in A}\v_G(a)$. In $G$, vertices $a\in A$ are adjacent to vertices in $A$ or to vertices in $\bigcup_{B\in N_G(A)}B\setminus A\subset V\setminus A$. We indicate this by
$\deg_G(a)= \deg_G(a,A)+\deg_G(a,V\setminus A)$. Moreover, we have that $\deg_{F_k}(A)=\sum_{a\in A}\deg(a,V\setminus A)$.

Then, Equation \eqref{eq:1} reads as
\begin{equation}
\label{eq:2}
\deg_{F_k}(A)\sum_{a\in A}\v_G(a)\ -\sum_{B\in N_{F}(A)}\sum_{b\in B}\v_G(b)=\lambda \sum_{a\in A}\v_G(a).
\end{equation}

For $B\in  N_{F}(A)$, we know that $A$ and $B$ share $k-1$ tokens. That is, $B =(A\setminus\{a\}) \cup \{b\}$ for some $a\in  A$ and $b \in B\setminus A$.

Then, we can write the second term in \eqref{eq:2} as
$$
\sum_{B\in N_{F}(A)}\sum_{b\in B}\v_G(b)=\sum_{a\in A}\sum_{\stackrel{b\in N_G(A)}{b\sim a}}\v_G(b)
+\sum_{a\in A}\v_G(a)[\deg_{F_k}(A)-\deg_G(a,V\setminus A)].
$$
To show this, we split the elements of $B$ into those that also belong to $A$, and the remaining one. For $b\in N_G(A)$, we count $\deg_G(b, A)$ times
$\v_G(b)$. This gives the first term. For the second term, for each element $a\in A$,
we count $\v_G(a)$ whenever another element of $A$, different from $a$, is adjacent to
an element $b\in N_G(A)$. The number of edges between $A$ and
$V\setminus A$ is $\deg_{F_k}(A)$, from which we subtract the number of edges from $a$ to $V\setminus A$.

Then, we can rewrite \eqref{eq:2} as
\begin{equation}
\label{eq:3}
\sum_{a\in A}\v_G(a)\deg_G(a,V\setminus A)-\sum_{a\in A}\sum_{\stackrel{b\in N_G(A)}{b\sim a}}\v_G(b)=\lambda\sum_{a\in A}\v_G(a).
\end{equation}

Now, if we add $\sum_{a\in A}\v_G(a)\deg_G(a,A)$ to the first term in \eqref{eq:3}, and subtract it from the second term, we get
\begin{equation}
\label{eq:4}
\sum_{a\in A}\v_G(a)\deg_G(a)-\sum_{a\in A}\sum_{b\sim a}\v_G(b)=\lambda\sum_{a\in A}\v_G(a).
\end{equation}

But, since $\v_G$ is an eigenvector of $\L$, for each $a\in A$, it holds that
$$
\v_G(a)\deg_G(a) - \sum_{b\sim a}\v_G(b)=\lambda \v_G(a).
$$
Thus, Equation \ref{eq:4} also holds and, consequently, $\v_F$ is indeed an eigenvector of $\L_k$ with associated eigenvalue $\lambda$.

Now, to prove that the Laplacian spectrum of $G$ is contained in the Laplacian spectrum of $F_k$, we have to show that independent eigenvectors of  $G$ give rise to independent eigenvectors in $F_k(G)$. With this aim, given some integers $n$ and $k$ (with $k\in [n]$), we define the $(n;k)$-\emph{binomial matrix} $\B$. This is a ${n \choose k}\times n$ matrix whose rows are the characteristic vectors of the $k$-subsets of $[n]$ in a given order. Thus, if the $i$-th subset is $A=\{a_1,\ldots,a_k\}$, then
$$
(\B)_{ij}=
\left\lbrace
\begin{array}{ll}
	1, & \mbox{if } a_j\in A,\\
	0, & \mbox{otherwise.}
\end{array}
\right.
$$
Notice that $\v_{F}=\B\v_G$, and, hence, we proved that $\B$ maps eigenvectors of $G$ associated with $\lambda$ to eigenvectors of $F_k(G)$ associated with $\lambda$. Hence, we only have to prove that such a mapping is injective or, equivalently, that $\Ker (\B)=\{\vec0\}$.
Let $\v=(v_1,\ldots,v_n)$ be such that $\B\v=\vec0$. Let $a,b\in [n]$ with $a\neq b$. We claim that $v_a=v_b$. Let $A_a\subset [n]$ be such that $|A_a|=k$, $A_a\ni a$, and $b\not\in A_a$. Let $A_b=(A_a\setminus\{a\})\cup \{b\}$.
Since
$$
(\B\v)_{A_a}=(\B\v)_{A_b}=0\quad\Rightarrow\quad \sum_{c\in A_a}v_c= \sum_{c\in A_b}v_c=0
$$
we have that
$$
v_a=-\sum_{c\in A_a\setminus \{a\}}v_c=-\sum_{c\in A_b\setminus \{a\}}v_c=v_b.
$$
Hence, there exists $\gamma\in \mathbb{R}$ such that $\v=\gamma\1$, where $\1$ is the all-$1$-vector and, from $\vec0=\B\v=k\v$, we have that $\v=\vec0$ and $\Ker (\B)=\{\vec0\}$. This completes the proof of the theorem.
\end{proof}

Figure \ref{fig2}
shows the construction of eigenvector $\v_F$ for the example graph $G$ of Figure \ref{fig1}
and eigenvalue $\lambda = 2$, which has associated eigenvector $\v_G = (1,0,-2,0, 1)$.

\begin{figure}[t]
\begin{center}
\vskip 1cm
\includegraphics[width=10cm]{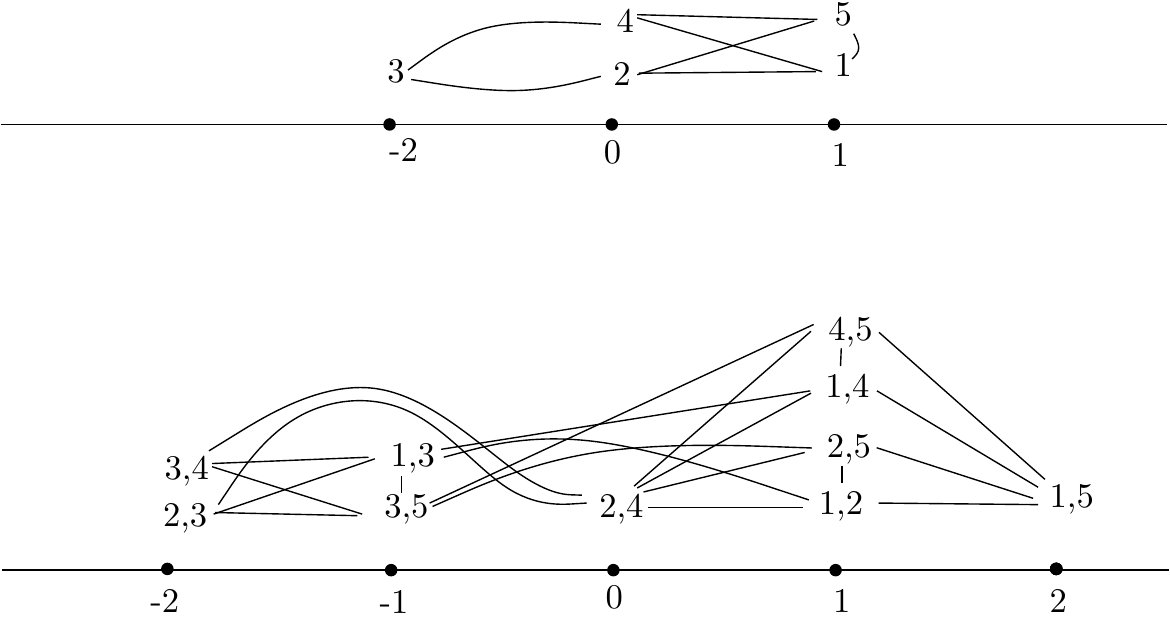}
\caption{Top: The eigenvector $\v_G$ for $\lambda = 2$ in $G$ is represented on the real
line. Bottom: The eigenvector $\v_{F}$ for $\lambda = 2$ in $F_2(G)$ is represented on the real line. Vertices of $F_2(G)$ are indicated by the positions of tokens on $G$. The edges of the graphs are also drawn.}
\label{fig2}
\end{center}
\end{figure}

\subsection{A matrix approach}
\label{sec:-3.1}
The following matrix approach provides an 
alternative compact formulation of the previous result.
First, let us recall a useful construction of the Laplacian matrix.
Given a graph $G=(V,E)$, with $n$ vertices and $m$ edges, consider a given {\em orientation} of it, that is, every edge $e=\{u,v\}\in E$ is replaced by an arc, say $(u,v)$. Then, we say that $u$ is a {\em positive end} of $e$, and $v$ is a {\em negative end} of $e$. The {\em incidence matrix} of $G$, with respect to a given orientation of it, is the
$n\times m$ matrix $\T=(t_{ij})$ with entries
$$
(\T)_{ij}=\left\{
\begin{array}{rl}
+1 & \mbox{if the vertex $u_i$ is a positive end of the edge $e_j$,}\\
-1 & \mbox{if the  vertex $u_i$ is a negative end of the edge $e_j$,}\\
0 & \mbox{otherwise}.\\
\end{array}
\right.
$$
Thus, each column of $\T$ has only two non-zero entries, $+1$ and $-1$. It is well known that, if $G$ is connected, then $\rank (\T)=n-1$ and the Laplacian of $G$ can be written as $\L=\T\T^{\top}$ (independently of the orientation). For more details about the properties of the incidence matrix, see, for instance, Biggs \cite[Ch. 4]{biggs}.

The following lemma is used to obtain the  result.
\begin{lemma}
	\label{lemma:1}
	The $(n;k)$-binomial matrix $\B$ satisfies
	$$
	\B^{\top} \B = {{n-2} \choose {k-1}}\I + {{n-2} \choose {k-2}} \J.
	$$
\end{lemma}

\begin{proof}
The diagonal entries  of $\B^{\top} \B$ are the number of 1's of each column of $\B$, that is, ${{n} \choose {k}}\frac{k}{n}={{n-1} \choose {k-1}}$. Moreover, the out-diagonal entries of $\B^{\top} \B$ are the common number of 1's between any two different columns of $\B$, which is ${{n-2} \choose {k-2}}$. Consequently, $\B^{\top} \B=\alpha \I+\beta \J$, with $\beta={{n-2} \choose {k-2}}$ and
	$$
	\alpha={{n-1} \choose {k-1}}-{{n-2} \choose {k-2}}={{n-2} \choose {k-1}},
	$$
	as claimed.
\end{proof}

Let $G$ be a graph with $n$ vertices and, for $k\leq \frac{n}{2}$, let $F_k=F_k(G)$ be its $k$-token graph. The following result gives the relationship between the corresponding Laplacian matrices, $\L_1$ and $\L_k$.

\begin{theorem}
\label{theo:main-result}
Given a graph $G$ and its $k$-token graph $F_k$, with corresponding Laplacian matrices $\L_1$ and $\L_k$, and $(n;k)$-binomial matrix $\B$,  the following holds:
\begin{equation}
\label{eq:main-result}
\B^{\top} \L_k \B = {{n-2} \choose {k-1}} \L_1.
\end{equation}
\end{theorem}

\begin{proof}
Let us first give a combinatorial proof of the result.
With this aim, the following simple fact from the definition of $F_k$ is relevant.
\begin{itemize}
\item[{\bf F1.}]
For any $a,b\in [n]$, the edges $\{A,B\}$ of $F_k$, such that $a\in A$ and $b\in B$ are necessarily of the form $\{\{a\}\cup X, \{b\}\cup X\}$, where $|X|=k-1$ and $a\sim b$ in $G$.   Thus, the number of such edges is ${n-2\choose k-1}$, which is precisely the coefficient of $\L_1$ in \eqref{eq:main-result}.
From this, we can state that each edge of $G$ becomes ${n-2\choose k-1}$ edges in $F_k(G)$.
\end{itemize}
Now, for any given orientation of $F_k$, consider its incidence matrix $\T_k$, so that
$\L_k=\T_k\T_k^{\top}$, and the left-hand term of \eqref{eq:main-result} becomes
\begin{equation}
\label{Ck}
\B^{\top}\T_k\T_k^{\top} \B=\C_k\C_k^{\top},
\end{equation}
where $\C_k=\B^{\top}\T_k$. Now, to prove the result, it suffices to show that $\C_k$ is a kind of `multiple incidence matrix' of $G$. More precisely, $\C_k$ corresponds to the incidence matrix of $G$, where each column has been repeated ${n-2\choose k-1}$ times (in some order and possibly interchanging its entries $\pm1$).
Indeed, for a given vertex $a$ of $G=(V,E)$ and edge ${\cal E}$ of $F_k=(V_k,E_k)$, we have
$$
(\C_k)_{a{\cal E}}=\sum_{A\in V_k}(\B^{\top})_{aA}(T_k)_{X{\cal E}}=
\left\{
\begin{array}{ll}
0 & \mbox{if $a\not\in A$,}\\
\displaystyle\sum_{A\ni a}(T_k)_{A{\cal E}} & \mbox{otherwise.}
\end{array}
\right.
$$
Moreover, in the second case, the edge ${\cal E}$ (seen as a $(k+1)$-subset of $[n]$) is of the form ${\cal E}=\{a\}\cup X\cup \{b\}$ for some $X\subset [n]$ with $|X|=k-1$ and $a,b\in [n]$, $b\neq a$. Then, depending on $b$, we have the following different possibilities:
\begin{enumerate}
\item
If $b\sim a$ in $G$ and $A=\{a\}\cup X$ is a positive end of the given orientation of $F_k$, then
$(\C_k)_{a{\cal E}}=1+0=1$;
\item
If $b\sim a$ in $G$ and $A=\{a\}\cup X$ is a negative end of the same orientation of $F_k$, then
$(\C_k)_{a{\cal E}}=-1+0=-1$;
\item
If $b\sim c$ in $G$ for some $c\in X$, then  $(C_k)_{a{\cal E}}=+1-1=0$.
\end{enumerate}
Besides, for fixed vertices $a,b$ of $G$ such that $a\sim b$, cases 1 and 2 appears as many times as ${\cal E}=\{a\}\cup X\cup \{b\}$, that is,  ${n-2\choose k-1}$ times (Fact {\bf 1}).
In other words, for every of the ${n-2\choose k-1}$ (oriented) edges ${\cal E}=\{a\}\cup X\cup \{b\}$ of $F_k(G)$, we have a column of $\C_k$ corresponding to an edge $\{a,b\}$ of $G$ with exactly two non-zero entries, $+1$ and $-1$, as required.


Alternatively, we can also prove \eqref{eq:main-result}  by using Theorem \ref{th:(1,k)}.
Let $\U$ be the $n\times (n-1)$ matrix whose columns are the normalized eigenvectors of $\L_1$ except the all-1 vector.
Let $\D$ be the diagonal $(n-1)\times (n-1)$ matrix of the corresponding eigenvalues. Then,

\begin{align*}
\eqref{eq:main-result} \ \B^{\top} \L_k \B = {{n-2} \choose {k-1}} \L_1 & \quad \Leftrightarrow \quad \B^{\top} \L_k \B \U = {{n-2} \choose {k-1}} \L_1 \U \\
 &  (\mbox{since } \rank (\B)=n,\ \rank (\U) =\rank (\L_1)=n-1),\\
		& \quad \Leftrightarrow \quad \B^{\top} \B \U \D = {{n-2} \choose {k-1}} \U \D \quad (\mbox{by Theorem } \ref{th:(1,k)}),\\
		& \quad \Leftrightarrow \quad \left({{n-2} \choose {k-1}} \I + {{n-2} \choose {k-2}} \J\right)  \U  \D = {{n-2} \choose {k-1}} \U \D\\
 &  (\mbox{by Lemma } \ref{lemma:1}).
\end{align*}
The last equality holds because $\J \U  \D =\O$,
since all columns of $\U $ are orthogonal to the all-1 eigenvector corresponding to the eigenvalue 0.
\end{proof}


\begin{corollary}
\label{coro1}
Given a graph $G(\cong F_1)$ and its $k$-token graph $F_k$, with corresponding Laplacian matrices $\L_1$ and $\L_k$, and $(n;k)$-binomial matrix $\B$,  the following implications hold:
\begin{itemize}
\item[$(i)$]
 If $\v$  is a $\lambda$-eigenvector of $ \L_1$, then $ \B\v$ is a  $\lambda$-eigenvector of $\L_k$.
\item[$(ii)$]
If $\w$  is a $\lambda$-eigenvector of $\L_k$ and $\B^{\top}\w\neq \vec0$, then  $\B^{\top}\w$  is a  $\lambda$-eigenvector of $\L_1$.
\end{itemize}
\end{corollary}

\begin{proof}
Statement $(i)$ corresponds to Theorem \ref{th:(1,k)}. Another proof, using Theorem \ref{theo:main-result}, goes as follows. We proved that
$\B^{\top} \L_k \B \U={{n-2} \choose {k-1}} \U \D$. Then, multiplying both terms by $\U ^{\top}$, we get that
$$
(\B\U)^{\top}\L_k(\B\U)={{n-2} \choose {k-1}} \D,
$$
which can be easily extended to a diagonalization of $\L_k$ (by adding the normalized $1$-vector to $\U $) since, by using again Lemma  \ref{lemma:1},
\begin{equation}
\label{ortho}
(\B\U)^
{\top}(\B\U)=\U^{\top}\B^{\top}\B\U={{n-2} \choose {k-1}}\I.
\end{equation}

The case $(ii)$ is proved later in Corollary \ref{coro2}.
\end{proof}

\begin{corollary}
	\label{coro:LkL1}
\begin{itemize}
\item[$(i)$]
The Laplacian spectrum (eigenvalues and their multiplicities) of $\L_1$ is contained in the Laplacian spectrum of $\L_k$.
\item[$(ii)$]
Every eigenvalue $\lambda$ of $\L_k$, having eigenvector $\w$ such that $\B^{\top}\w\neq \vec0$,  is a $\lambda$-eigenvector of   $\L_1$.
\end{itemize}
\end{corollary}

\begin{proof}
Statement $(i)$ is a consequence of Theorem  \ref{theo:main-result} and \eqref{ortho}, which, in fact, proves that a set of mutually orthogonal eigenvectors of $G$, different from the all-1 vector, gives rise to a set of mutually orthogonal eigenvectors of $F_k$.
Finally, $(ii)$ is just a reformulation of Corollary \ref{coro1}$(ii)$.
\end{proof}

\begin{figure}[t]
	\begin{center}
		\includegraphics[width=14cm]{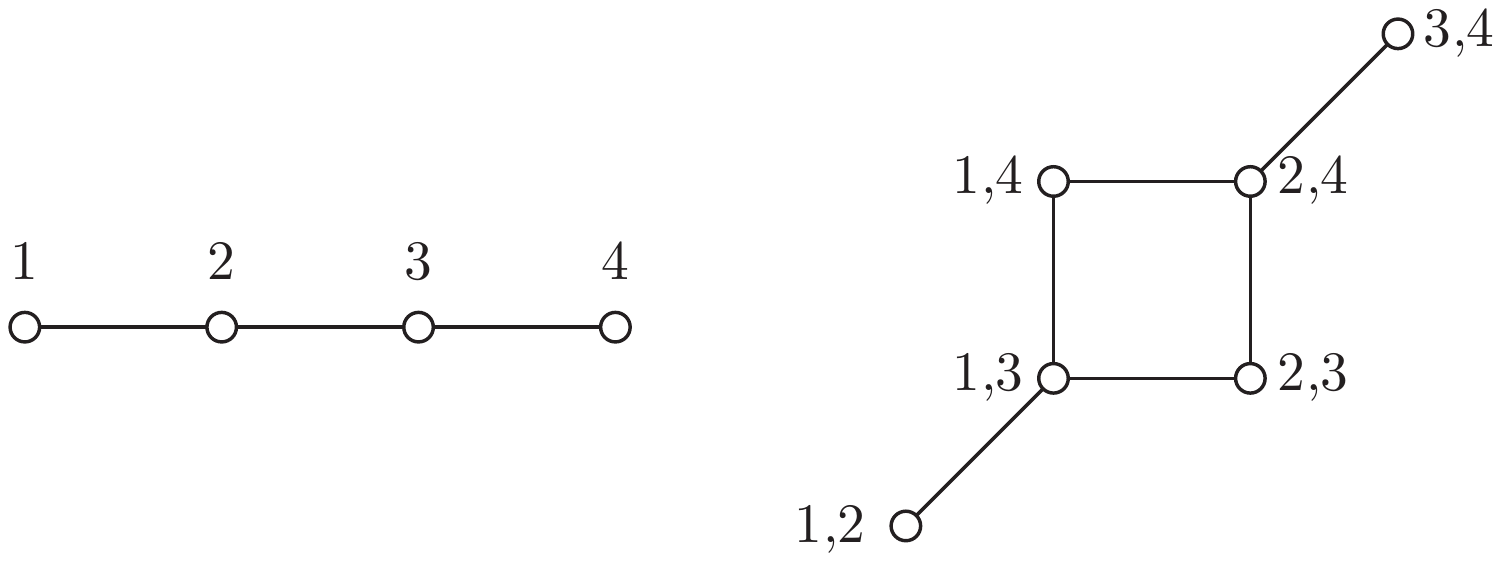}
	\end{center}
	\vskip-16cm
	\caption{The graph $G=P_4$ (left) and it 2-token graph $F_2(G)$ (right).}
	\label{fig:P_4+2token}
\end{figure}

\begin{example}
	\label{ex:1}
	Given the graph $G=P_4$ (the path on 4 vertices), we construct its 2-token graph $F_2(G)$ (see Figure \ref{fig:P_4+2token}). The Laplacian spectra of $G$ and $F_2(G)$ are $\{0,2-\sqrt{2},2,2+\sqrt{2}\}$ and
	$\{0,2-\sqrt2,3-\sqrt3,2,2+\sqrt2,3+\sqrt3\}$. The Laplacian matrices $\L_1$ and $\L_2$ are
	$$
	\L_1=\left(\begin{array}{rrrr}
		1&-1&0&0\\
		-1&2&-1&0\\
		0&-1&2&-1\\
		0&0&-1&1
	\end{array} \right), \qquad
\L_2=\left(\begin{array}{rrrrrr}
	2&-1&0&-1&0&0\\
	-1&3&-1&0&-1&0\\
	0&-1&2&-1&0&0\\
	-1&0&-1&3&0&-1\\
	0&-1&0&0&1&0\\
	0&0&0&-1&0&1
\end{array} \right).
	$$
	The binomial matrix $\B$ is
	$$
	\B=\left(\begin{array}{rrrr}
		1&1&0&0\\
		1&0&1&0\\
		1&0&0&1\\
		0&1&1&0\\
		0&1&0&1\\
		0&0&1&1
	\end{array} \right).
	$$
	We can check that $\B^{\top} \L_2 \B = {4-2 \choose 2-1} \L_1= 2 \L_1$.
\end{example}

\begin{figure}[t]
	\begin{center}
		\includegraphics[width=14cm]{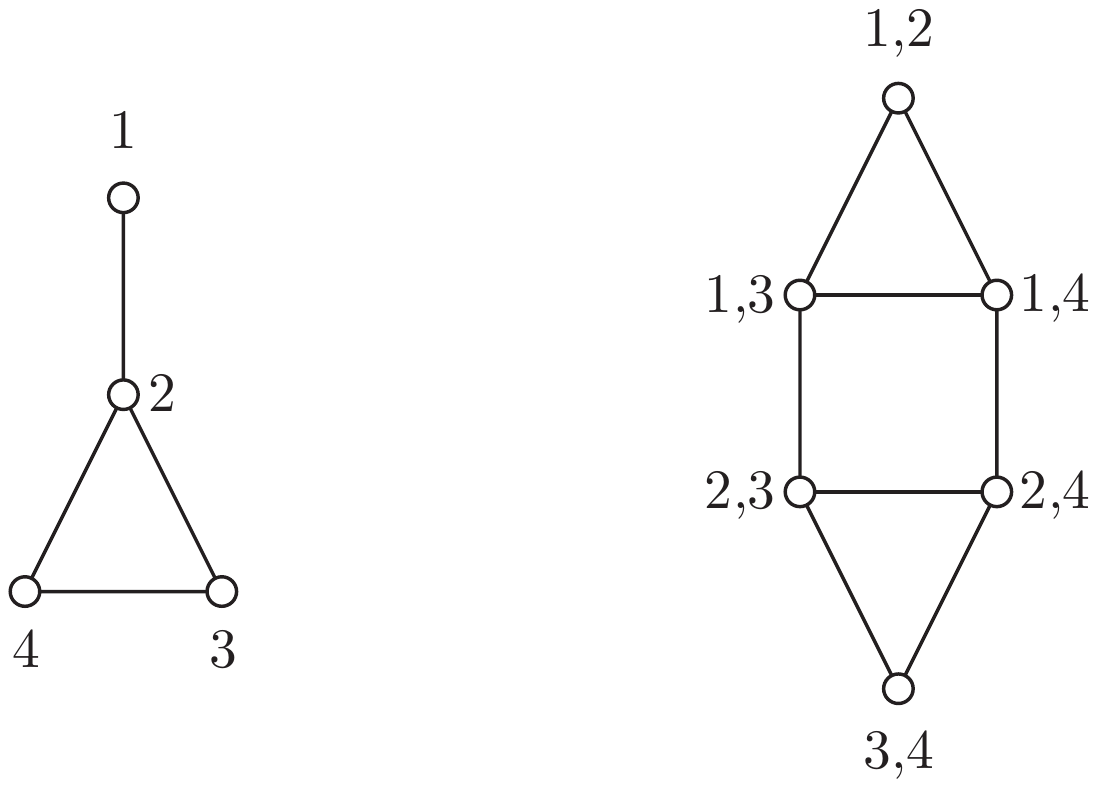}
	\end{center}
	\vskip-14.5cm
	\caption{A graph $G$ (left) and its 2-token graph $F_2(G)$ (right).}
	\label{fig:C_3-amb-aresta+2token}
\end{figure}


\section{A more general result}
\label{sec:-4}
In this section, we prove a stronger result. Namely, for any $1\le k< n/2$, the Laplacian spectrum of the $k$-token graph $F_k(G)$ of a graph $G$ is contained in the Laplacian spectrum of its $(k+1)$-token graph $F_{k+1}(G)$.

\subsection{The local approach}
\label{sec:-4.1}
We begin by proving that every eigenvalue  of $F_k(G)$ is also an eigenvalue of $F_{k+1}(G)$ through a `local analysis', as in the proof of Theorem \ref{th:(1,k)}.

\begin{theorem}
\label{th:(h,k)}
Let $G$ be a graph on $n$ vertices. Let $h,k$ be integers such that $1\le h\le k\le n/2$.
If $\lambda$ is an eigenvalue of $F_{h}(G)$, then $\lambda$ is an eigenvalue of $F_{k}(G)$.
\end{theorem}
\begin{proof}
It suffices to prove the result for $F_k:=F_k(G)$ and $F_{k+1}:=F_{k+1}(G)$, where $1\le k\le n/2-1$.
Let $\v_k$ be an eigenvector of $F_k$ with eigenvalue $\lambda$. Then, we define $\v_{k+1} $ in the following way:
$$
\v_{k+1}(A):=\sum_{X\in V(F_k): X\subset A} \v_k(X)\quad\mbox{for every vertex $A$ of $F_{k+1}$.}
$$
We want to show that $\v_{k+1}$ is an eigenvector of $F_{k+1}$ with eigenvalue $\lambda$.
Let $\L_k$ and $\L_{k+1}$ be the Laplacian matrices of $F_k$ and $F_{k+1}$, respectively.
Given a vertex $A$ of $F_{k+1}$, let $\L_{k+1}(A)$ denote the row of $\L_{k+1}$ corresponding to $A$. Moreover, let $J_A:=\{X\in V(F_k):X\subset A\}$. Then,
\begin{align}
\L_{k+1}(A)\v_{k+1} &= \deg_{F_{k+1}}(A)\v_{k+1}(A)-\sum_{\stackrel{B\in V(F_{k+1})}{B\sim A}} \v_{k+1}(B) \nonumber\\
 & = \sum_{\stackrel{B\in V(F_{k+1})}{B\sim A}} [\v_{k+1}(A)-\v_{k+1}(B)] \nonumber\\
 & = \sum_{\stackrel{B\in V(F_{k+1})}{B\sim A}} \Big[\sum_{X\in J_A} \v_k(X)-\sum_{Y\in J_B} \v_k(Y)\Big] \label{step0}\\
& = \sum_{X\in J_A}
\sum_{Y\sim X } [\v_k(X)-\v_k(Y))]   \label{key-step}\\
&=\sum_{X\in J_A}\deg_{F_k}(X)\v_k(X)-\sum_{Y\sim X}\v_k(Y) \nonumber
\end{align}
\begin{align}
 & = \sum_{X\in J_A}\lambda\v_k(X)=\lambda\v_{k+1}(A). \nonumber
\end{align}

The reason for Equation \eqref{key-step} is the following. First, notice that, for each $B\sim A$, we have that $A=Z\cup\{a\}$ and $B=Z\cup\{b\}$, where $Z=A\cap B$, and $a\sim b$ in $G$ (recall that $|X|=|Y|=|A\cap B|$).
Now, if either $X,Y\subset A$ or $X,Y\subset B$ (including the possible cases where $X=Y=Z$ or $Y\sim X$), both terms $\v_k(X)$ and $\v_k(Y)$ appear in the second and third sum of \eqref{step0}, giving zero.
Otherwise, for each $X\in J_A$ such that $X\ni a$, there is one $Y\in J_B$ such that $Y\ni b$,
$X\bigtriangleup Y= A\bigtriangleup B$, and $Y\sim X$.

Consequently, $\v_{k+1}$ is an eigenvector of $F_{k+1}$ with eigenvalue $\lambda$.
\end{proof}

\subsection{The matrix approach}
\label{sec:-4.2}
All previous results can be seen as consequences of the following matricial formulation. First, we  define, for some integers $n$, $k_1$, and $k_2$ (with $1\le k_1<k_2<n$), the $(n;k_2,k_1)$-\emph{binomial matrix} $\B=\B(n;k_2,k_1)$. This is a ${n \choose k_2}\times {n \choose k_1}$ $(0,1)$-matrix, whose rows are indexed by the $k_2$-subsets $A\subset [n]$,
and its columns are indexed by the $k_1$-subsets $X\subset [n]$. The entries of $\B$ are
$$
(\B)_{AX}=
\left\lbrace
\begin{array}{ll}
	1 & \mbox{if } X\subset A,\\
	0 & \mbox{otherwise.}
\end{array}
\right.
$$
The transpose of $\B=\B(n;k_2,k_1)$ is known as the {\em set-inclusion matrix}, denoted by $W_{k_1,k_2}(n)$ (see, for instance, Godsil \cite{g95}).
\begin{lemma}
\label{lemma:2}
The matrix $\B$ satisfies the following simple properties.
\begin{itemize}
\item[$(i)$]
The number of 1's of each column of $\B$ is ${n-k_1\choose k_2-k_1}$.
\item[$(ii)$]
The common number of 1's of any two columns of $\B$, corresponding to $k_2$-subsets of $[n]$ whose intersection has $k_1-1$ elements, is ${n-k_1-1 \choose k_2-k_1-1}$.
\end{itemize}
\end{lemma}
\begin{proof}
$(i)$ The number of 1's of each column can be computed as (\#rows)$\times$(\#1's per row)/(\# columns), which gives
$
{n\choose k_2}{k_2\choose k_1}/{n\choose k_1}={n-k_1\choose n-k_2}={n-k_1\choose k_2-k_1}.
$
Alternatively, notice that this is just the number of $k_2$-subsets of $[n]$ containing a given $k_1$-subset of $[n]$.
Similarly, the number of 1's in $(ii)$ equals the number of $k_2$ subsets containing a $(k_1+1)$-subset of $[n]$.
\end{proof}
The new matrix $\B$ allows us to give the following result that can be seen as a generalization of Theorem \ref{theo:main-result} (see also Corollary \ref{coro:Lk-L1}).

\begin{theorem}
\label{theo:general-result2}
Let $G$ be a graph on $n=|V|$ vertices, with $k_1$- and $k_2$-token graphs $F_{k_1}(G)$ and $F_{k_2}(G)$, where $1\le k_1\le k_2\le n$. Let $\L_{k_1}$ and $\L_{k_2}$ be the respective Laplacian matrices, and $\B$ the $(n;k_2,k_1)$-binomial matrix. Then,  the following holds:
\begin{equation}
\label{eq:general-result2}
\B \L_{k_1} =  \L_{k_2}\B.
\end{equation}
\end{theorem}
\begin{proof}
Let  $X,X',\ldots$ be vertices of
$F_{k_1}(G)$ and let $A,A',\ldots$ be vertices of $F_{k_2}(G)$, seen as $k_1$- and $k_2$-subsets of $[n]$, respectively. Then, we want to prove that $(\B \L_{k_1})_{AX}=(\L_{k_2}\B)_{AX}$ for every $A\in V_{k_2}$ and $X\in V_{k_1}$. The proof is based on considering the different cases of the intersection of  the set $X$ with $A$.
The first term is (the role of $a\in [n]\simeq V(G)$ is explained afterward):
\begin{align*}
(\B \L_{k_1})_{AX}&=\sum_{X'\in V_{k_1}} (\B)_{AX'}(\L_{k_1})_{X'X}=\sum_{X'\subset A} (\L_{k_1})_{X'X}\\
&=\left\lbrace
\begin{array}{cl}
	0 & \mbox{if \  } |X\cap A|<k_1-1,\qquad\qquad $(a1)$\\
	-|N_G(a)\cap (A\setminus X)| & \mbox{if \  } |X\cap A|=k_1-1,\qquad\qquad $(a2)$\\
	\deg_{F_{k_1}(G)}(X)- |E_G(X,A\setminus X)|  & \mbox{if \  } |X\cap A|=k_1\ (X\subset A),\quad \ \ $(a3)$\\
\end{array}
\right.
\end{align*}
while the second term is:
\begin{align*}
(\L_{k_2}\B)_{AX}&=\sum_{A'\in V_{k_2}} (\L_{k_2})_{AA'}(\B)_{A'X}=\sum_{A'\supset X} (\L_{k_2})_{AA'}\\
&=\left\lbrace
\begin{array}{cl}
	0 & \mbox{if \ } |X\cap A|<k_1-1,\qquad\qquad $(b1)$\\
	-|N_G(a)\cap (A\setminus X)| & \mbox{if \  } |X\cap A|=k_1-1,\qquad\qquad $(b2)$\\
	\deg_{F_{k_2}(G)}(A)- |E_G(A\setminus X,V\setminus A)| & \mbox{if  \ } |X\cap A|=k_1 \ (X\subset A).\quad \ \ $(b3)$
\end{array}
\right.
\end{align*}
Now, let us justify each  of the above equalities, and why $(a3)=(b3)$:
\begin{enumerate}
\item
If $|X\cap A|<k_1-1$, then $(a1)$ and $(b1)$ are zero because neither $X'\subset A$ nor  $A'\supset X$ can be adjacent to $X$ and $A$, respectively.
\item
If $|X\cap A|=k_1-1$, then $X$ is of the form $X=(X\cap A)\cup \{a\}$, where $a\in [n]$ and $a\not\in A$. Then, the expression in $(a2)$ is because $X'\subset A$ can be any set of the form $\{b\}\cup(A\cap X)$, where $b\sim a$ and  $b\in A\setminus X$.
Similarly, the value of $(b2)$ is because every $A'\supset X$ is of the form $(A\setminus \{b\})\cup \{a\}$, where $b$ is as before.
\item
If $X\subset A$, there are two kinds of contributions to $\sum_{X'\subset A}(\L_{k_1})_{X'X}$, which gives $(a3)$.
First, when $X'=X$, we get the term $\deg_{F_{k_1}(G)}(X)$. Second, if $X'\neq X$, then we get one `$-1$' for each $X'$.
In this case, every subset $X'\subset A$ such that $X'\sim X$ in $F_{k_1}(G)$ is of the form $(X\setminus \{a\})\cup \{b\}$, where $a\in X$ and  $b\in A\setminus X$. Thus, the total number of such subsets $X'$ is $|E_G(X,A\setminus X)|$.
Analogously, the value in $(b3)$ contains two terms. When $A'=A$, we have $(\L_{k_2})_{A'A}=\deg_{F_{k_2}(G)}(A)$. Otherwise, every $A'\supset X$ such that $A'\sim A$ in $F_{k_2}(G)$ is of the form $(A\setminus \{b\})\cup \{a\}$, where $b\in A\setminus X$ and $a\in V\setminus A$. Now, the number of such subsets $A'$ is $|E_G(A\setminus X,V\setminus A|$.
\end{enumerate}
Finally, notice that both terms $(a3)$ and $(b3)$ are equal since
\begin{align}
\deg_{F_{k_1}}(X)- |E_G(X,A\setminus X)|
 & = |E_G(X,V\setminus X)|-|E_G(X,A\setminus X)| \nonumber\\
& = |E_G(X,V\setminus A)| \label{eq:setsAX}\\
 & = |E_G(A,V\setminus A)|-|E_G(A\setminus X,V\setminus A)| \nonumber\\
 & =\deg_{F_{k_2}}(A)- |E_G(A\setminus X, V\setminus A)| \nonumber.
\end{align}
For a better understanding of the above equalities, see Figure \ref{fig-dibu-teorema}, where we use the following classes of edges in $G$:\\
\mbox{\hskip .5cm}
$-$ $E_G(X,V\setminus X)$ are the thick and the thin edges.\\
\mbox{\hskip .5cm}
$-$ $E_G(X,A\setminus X)$ are the thin edges.\\
\mbox{\hskip .5cm}
$-$ $E_G(A,V\setminus A)$ are the thick and the dashed edges.\\
\mbox{\hskip .5cm}
$-$ $E_G(A\setminus X,V\setminus A)$ are the dashed edges.\\
\mbox{\hskip .5cm}
$-$ The edges represented in dotted lines are those not in any of the four previous sets.\\
This completes the proof.
\end{proof}

\begin{figure}[t]
	\begin{center}
		\includegraphics[width=12cm]{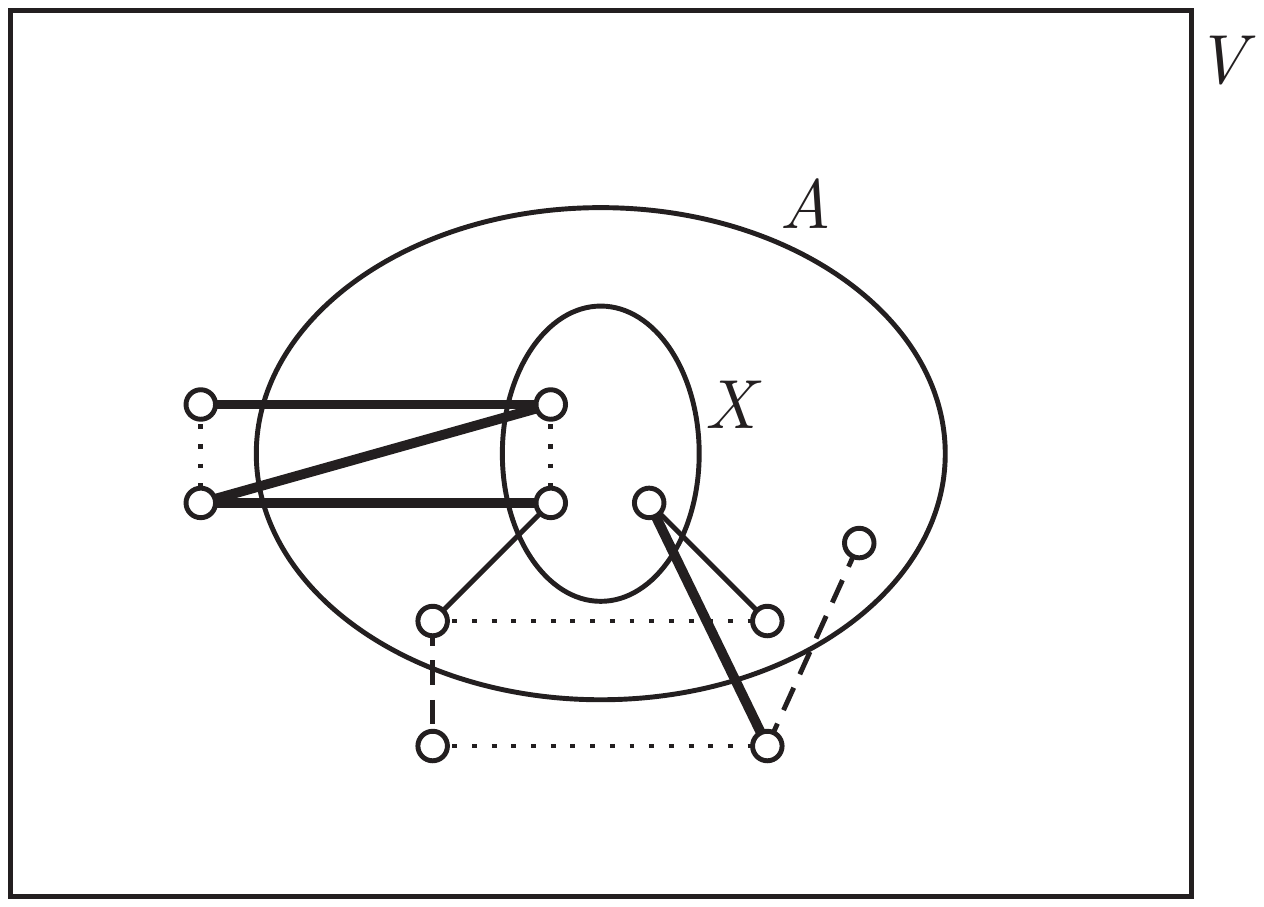}
		\vskip-11.25cm
		\caption{Scheme of the different kind of edges of $G$ in \eqref{eq:setsAX}.}
		\label{fig-dibu-teorema}
	\end{center}
\end{figure}

Let us now see some consequences of this theorem. First, we get again Theorem \ref{theo:main-result}.
\begin{corollary}
\label{coro:Lk-L1}
With the same notation as before, the following holds.
\begin{itemize}
\item[$(i)$]
For every $k_1,k_2$ with $1\le k_1\le k_2\le n$,
\begin{equation}
\label{eq:Lk1-Lk2}
\B^{\top}\L_{k_2}\B=\B^{\top}\B\L_{k_1}.
\end{equation}
\item[$(ii)$]
For	$k_1=1$ and $k_2=k$,
$$
\B^{\top} \L_k \B = {{n-2} \choose {k-1}} \L_1.
$$
\end{itemize}
\end{corollary}
\begin{proof}
$(i)$ Follows directly from \eqref{eq:general-result2} multiplying by $\B^{\top}$. The result in $(ii)$, which corresponds to Theorem \ref{theo:main-result}, is obtained from $(i)$ with $k_1=1$ and $k_2=k$ and Lemma  \ref{lemma:1}. Indeed,
\begin{align*}
 \B^{\top}\B \L_{1} =  \B^{\top} \L_{k}\B
 & \Longleftrightarrow  \left[{{n-2} \choose {k-1}}\I + {{n-2} \choose {k-2}} \J\right] \L_{1} =  \B^{\top}\L_{k}\B \\
  & \Longleftrightarrow {{n-2} \choose {k-1}} \L_{1}=\B^{\top} \L_{k}\B
\end{align*}
since $\J\L_{1}=\O$.
\end{proof}

Since $F_k(G)\cong F_{n-k}(G)$ assume, without loss of generality, that $1\le k_1\le k_2 \le \frac{n}{2}$. Then, we have the generalization of Corollary \ref{coro:Lk-L1}.
\begin{corollary}
\label{coro2}
For any integers $h,k$ such that $1\le h\le k \le \frac{n}{2}$, let $\B$ be the $(n;k,h)$-binomial matrix. Then, the eigenvalues and eigenvectors of the Laplacian matrices of the token graphs $F_h$ and $F_k$ are related in the following way.
\begin{itemize}
\item[$(i)$]
If $\v$ is a $\lambda$-eigenvector of $\L_{h}$, then $\B \v$ is a  $\lambda$-eigenvector of $\L_{k}$. Moreover, the linear independence of the different eigenvectors is preserved. (That is, the spectrum of $\L_h$ is contained in the spectrum of $\L_k$.)
\item[$(ii)$]
If $\w$  is a $\lambda$-eigenvector of $\L_{k}$ and $\B^{\top}\w\neq \vec0$, then  $\B^{\top}\w$  is a  $\lambda$-eigenvector of $\L_{h}$. Moreover, all the eigenvalues, including multiplicities, of $\L_h$ are obtained (that is, one eigenvalue each time that the above non-zero condition is fulfilled).
\end{itemize}
\end{corollary}
\begin{proof}
$(i)$ Let $\U$ be now the ${n \choose h}\times {n \choose h}$ matrix whose columns are all the normalized eigenvectors of $\L_h$. Then, Theorem \ref{theo:general-result2} with $k_1=h$ and $k_2=k$ yields
\begin{equation}
\label{k-vs-k}
\L_k\B\U=\B\L_h\U=\B\lambda \U= \lambda \B\U,
\end{equation}
and the result follows since $\rank (\B(n;k,h))=\min\left\{{n\choose k},{n\choose h}\right\}={n\choose h}$ (see de Caen \cite{dc01}) implies that $\Ker (\B)=\{\vec0\}$ by the Rank-Nullity theorem.\\
$(ii)$ Let $\W=\U $ be the same matrix as in $(i)$. Since, in fact, \eqref{k-vs-k} holds for any $k\le n-1$, let us apply it by putting  $n-h$ instead of $k$, and $k$  instead of $h$ (so that now $k<n-h$). Then, we get
$$
\L_{n-h}\B^{\top}\W =\B^{\top}\L_k\W=\lambda \B^{\top}\W,
$$
since  $\B(n;n-h,k)=\B(n;k,h)^{\top}=\B^{\top}$, and the matrices $\L_h$ and $\L_{n-h}$ have the same spectrum (recall that $F_{n-h}\cong F_h$).
Finally, all the eigenvectors of $\L_h$ are obtained because $\rank (\B^{\top})=h$.
\end{proof}

In our context, Theorem \ref{theo:general-result2} allows us to obtain the Laplacian matrix of $F_h$ in terms of the Laplacian matrix of $F_k$, provided that we know the binomial matrix $\B(n;k,h)$  with its rows and columns in the right order (that is, the same order as the columns of $\L_h$ and $\L_k$, respectively). Indeed, in this case,  \eqref{eq:Lk1-Lk2} with $k_1=h$ and $k_2=k$ leads to
\begin{align}
\L_h &=(\B^{\top}\B)^{-1}\B^{\top}\L_k \B.
\end{align}
Notice that $\B^{\top}\B$ is a Gram matrix of the columns of $\B$, which are linearly independent vectors and, hence, $\B^{\top}\B$ has inverse matrix.

Following with the simplified notation $k_1=h$ and $k_2=k$, the result of Theorem \ref{theo:general-result2} can also be written in terms of the adjacency matrices $\A_h$ and $\A_k$ of $F_h$ and $F_k$, respectively. Then, we get
\begin{equation}
\label{AkAh}
\A_k\B-\B\A_h=\D_k\B-\B\D_h,
\end{equation}
where $\D_h$ and $\D_k$ are the diagonal matrices with entries the degrees of the vertices of $F_h$ and $F_k$, respectively.
Some consequences of this are obtained when both $F_h$ and $F_k$ are regular. In Carballosa, Fabila-Monroy, Lea\~nos, and Rivera \cite{cflr17}, it is was proved that this regularity condition holds in the following cases:  $G=K_n$ or  $G=\overline{K_n}$ and $2\le k\leq n-2$; $G=K_{1,n-1}$ (the $(n-1)$-pointed star with $n$ even) or $G=\overline{K_{1,n-1}}$, and $k=n/2$.
\begin{corollary}
Assume that a graph $G\equiv F_1$ and its $k$-token graph $F_k$ are $d_1$-regular and $d_k$-regular graphs, respectively. Let $\B$ be the $(n;k,1)$-binomial matrix. Let $\A$ and $\A_k$ be the respective adjacency matrices of $G$ and $F_k$. If $\v$ is a $\lambda$-eigenvector of $\A$, then $\B \v$ is a $\mu$-eigenvector of $\A_k$, where $\mu=(d_k-d_1+\lambda)$.
\end{corollary}
\begin{proof}
	Use \eqref{AkAh} with $\D_k=d_k\I$ and $\D_h=\D=d_1\I$, or consider the following implications:
	\begin{align*}
	\mbox{$\v$ is a $\lambda$-eigenvector of $\A$} & \Rightarrow 	\mbox{$\v$ is a $(d_1-\lambda)$-eigenvector of $\L$} \Rightarrow \\
	\mbox{$\B \v$ is a $(d_1-\lambda)$-eigenvector of $\L_k$} & \Rightarrow \mbox{$\B \v$ is a $[d_k-(d_1-\lambda)]$-eigenvector of $\A_k$}.
	\end{align*}
\end{proof}

\section{The Laplacian spectra of two infinite families of graphs}
\label{sec:-5}

In this section, we describe two infinite families of graphs in which their token graphs have the Laplacian spectrum fully characterized. In fact, one of these families and part of the other are the two mentioned infinite families of graphs that give rise to regular token graphs. Namely, the complete graphs $K_n$, for every $n$ and $k$; and the star graphs $S_n=K_{1,n-1}$, for $n$ even and $k=n/2$.

Before proceeding, we recall a useful graph construction and some of its properties.
Let $G=(V,E)$ be a graph of order $n$.
Its {\em double graph} $2G\equiv \widetilde{G}=(\widetilde{V},\widetilde{E})$
is the graph with the duplicated vertex set $\widetilde{V}=V\cup V'$,
and adjacencies induced from the adjacencies in $G$ as follows:
\begin{equation}
\label{def.doble.bipartit}
u\, \stackrel{\scriptscriptstyle{(E)}}{\sim} \, v\quad \Rightarrow\quad
u\, \stackrel{\scriptscriptstyle{(\widetilde{E})}}{\sim}\, v'
\quad \textrm{and}\quad
v\, \stackrel{\scriptscriptstyle{(\widetilde{E})}}{\sim}\, u'.
\end{equation}
Thus, the edge set of $\widetilde{G}$ is $\widetilde{E}=\{uv':uv\in
E \}$.
Observe that $2G$ is, in fact, the cross product of $G$ with $K_2$.
From the definition, it follows that $\widetilde{G}$ is a bipartite
graph with stable subsets $V$ and $V'$.
Moreover, if $G$ is a $\delta$-regular graph, then $\widetilde{G}$ also is.
For example,  Figure~\ref{fig.C5} shows the cycle $C_5$ and its double graph $C_{10}$.

\begin{figure}[t]
\centering
\includegraphics[scale=0.5]{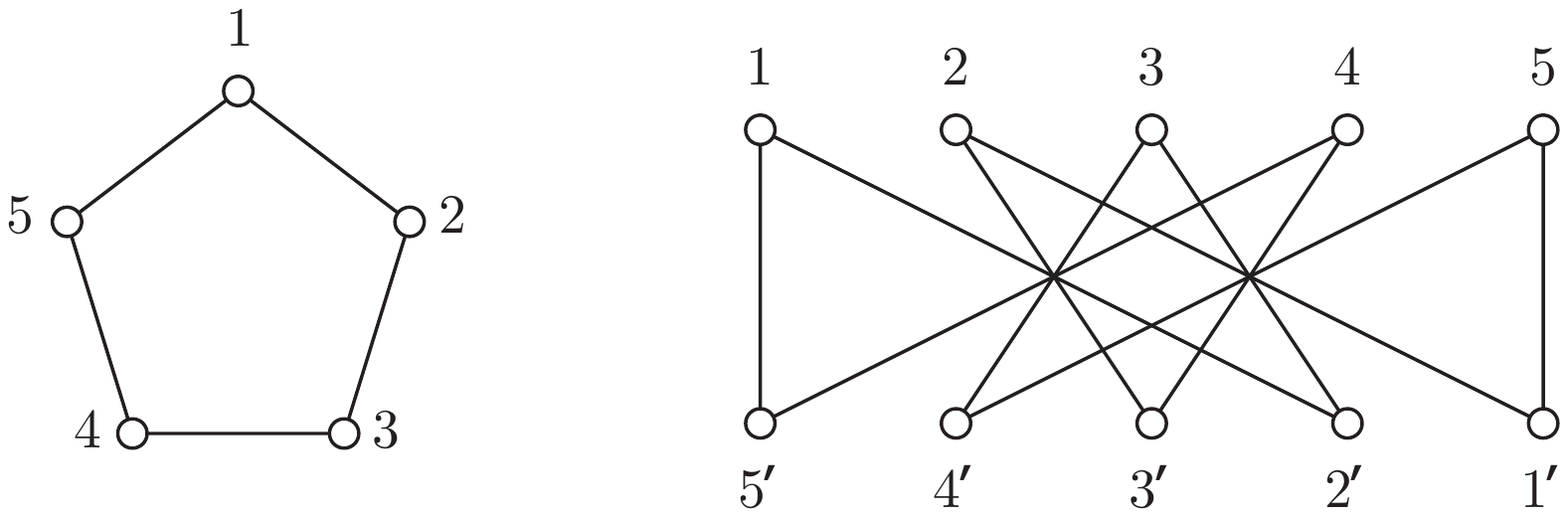}
\vskip-12cm
\caption{$C_{5}$ and its double graph $C_{10}$.}
\label{fig.C5}
\end{figure}

Concerning the spectral properties of the adjacency matrix, it is known that the eigenvalues and eigenvectors of the  double
graph $\widetilde{G}$ are closely related to the eigenvalues and eigenvectors of $G$ (see, for instance, Dalf\'o, Fiol, and Mitjana \cite{dfm15}). From now on, the eigenvalues of the adjacency matrix of a graph will be denoted by $\mu$'s (to clearly distinguish them from  the Laplacian eigenvalues, here denoted by $\lambda$'s).

\begin{proposition}
\label{prop.dc.espectrum}
Let $G$ be a graph on $n$ vertices.
\begin{itemize}
\item[$(i)$]
Let $G$ have a $\mu$-eigenvector $\x$.
Let us consider the vectors $\x^+$ with components
$x_i^+=x_{i'}^+=x_{i}$, and $\x^-$, with components
$x_i^-=x_i$ and $x_{i'}^-=-x_{i}$, for $1 \leq i,i' \leq n$.
Then, $\x^+$ is a $\mu$-eigenvector of $\widetilde{G}$ and $\x^-$ is a $(-\mu)$-eigenvector of $\widetilde{G}$.
\item[$(ii)$]
Let $G$ have spectrum
$\spec G=\{\mu_0^{m_0},\mu_1^{m_1}, \dots, \mu_d^{m_d} \}$,
where the superscripts denote multiplicities. Then, the spectrum of
$\widetilde{G}$ is
\begin{equation}
\spec\widetilde{G} = \{\pm \mu_0^{m_0},\pm \mu_1^{m_1},
\ldots, \pm\mu_d^{m_d} \}.
\end{equation}
\end{itemize}
\end{proposition}

Given two integers $n,k$ such that $k\in [n]$, the {\em Johnson graph} $J(n,k)$ is defined in the same way as the $k$-token graph of the complete graph $K_n$. Then, $F_k(K_n)\cong J(n,k)$. It is known that these graphs are antipodal (but not bipartite) distance-regular graphs, with degree $d=k(n-k)$, diameter $D=\min\{k,n-k\}$, and spectrum of the adjacency matrix $\A$
(eigenvalues and multiplicities)
$$
\mu_j=(k-j)(n-k-j)-j\qquad \mbox{and} \qquad m_j={n\choose j}-{n\choose j-1},\qquad j=0,1,\ldots,D.
$$
See, for instance, Brouwer, Cohen, and Neumaier \cite[Th. 9.1.2]{bcn89}. Then, since the Laplacian matrix of $J(n,k)$ is $\L_k=d\I-\A$, we infer that the Laplacian spectrum of $F_k(K_n)$ is
\begin{equation}
\lambda_j=d-\mu_j=j(n+1-j)\qquad \mbox{and} \qquad m_j={n\choose j}-{n\choose j-1},\qquad j=0,1,\ldots,D.
\end{equation}
For example, $F_7(K_{14})\cong J(14,7)$ is a $49$-regular graph with $n=3432$ vertices, diameter $D=7$, and Laplacian spectrum
$$
\spec F_7(K_{14})=\{0^1, 14^{13}, 26^{77}, 36^{273}, 44^{637}, 50^{1001}, 54^{1001}, 56^{429}\}.
$$

Given an integer $k$, the {\em odd graph} $O_k$ has vertices corresponding to the ${2k-1\choose k}$ $k$-subsets of the set $[2k-1]$, and two vertices are adjacent if the corresponding $k$-subsets are disjoint.
The adjacency eigenvalues and multiplicities of $O_k$ are
\begin{equation}
\label{sp-Ok}
\mu_j=(-1)^j(k-j)\qquad \mbox{and} \qquad m_j={2k-1\choose j}-{2k-1\choose j-1},\qquad j=0,1,\ldots,k-1,
\end{equation}
(see, for instance, Biggs \cite{biggs}[Chaps. 20-21]).
Now, consider the {\em double odd graph $2O_k$} with stable sets $V_0=\{(u,0):u\in V(O_k)\}$ and $V_1=\{(v,1): v\in V(O_k)\}$,
and $(u,0)\sim (v,1)$ when $u\sim v$ in $O_k$. Alternatively, since ${2k-1\choose k-1}={2k-1\choose k}$, the graph $2O_k$ can be defined by saying that its vertices are the $(k-1)$-subsets and $k$-subsets of the set $[2k-1]$, where distinct vertices $\gamma$ and $\delta$ are adjacent  if and only if either $\gamma \subset \delta$ or $\delta \subset \gamma$.
The double odd graph $2O_k$ is an antipodal (bipartite) distance-regular graph with degree $d=k$, order $2{2k-1\choose k-1}$, diameter $D=2k-1$, and adjacency spectrum
$$
\mu_{\pm j}=\pm(k-j)\qquad \mbox{and} \qquad m_{\pm j}={n\choose j}-{n\choose j-1},\qquad j=0,1,\ldots,k-1.
$$
This is a consequence of \eqref{sp-Ok} and Proposition \ref{prop.dc.espectrum}$(i)$ (see again \cite[pp. 414]{bcn89}).
Moreover, we claim that, when $n=2k$,  the $k$-token graph of the star $S_n$ is isomorphic to the double odd graph $2O_{k}$. Indeed, note first that the number of vertices of both graphs coincides since ${2k\choose k}={2k-1\choose k-1}+{2k-1\choose k}$. Moreover, with vertices of $S_{2k}$ labeled $0$ (the center) and $1,\ldots,2k-1$ (the spokes), the mapping that sends every vertex of $F_k(S_{2k})$ corresponding to a $k$-subset $\gamma\ni 0$ to the vertex $(0,\gamma\setminus{0})$ of $2O_{k}$, and every vertex of $F_k(S_{2k})$ corresponding to a $k$-subset $\gamma\not\ni 0$ to the vertex $(1,\gamma)$ of $2O_{k}$ is an isomorphism between $F_{k}(S_{2k})$ and $2O_k$.
As a consequence, we have that the $(n/2)$-token of the star graph $S_n$ has Laplacian spectrum
$$
\lambda_{j}=j,\quad \lambda_{-j}=2k-j,\qquad \mbox{and} \qquad m_{\pm j}={n\choose j}-{n\choose j-1},\qquad j=0,1,\ldots,k-1.
$$

The double odd graphs can be generalized as follows.
The {\em doubled Johnson graph} $J(n;k,k+1)$ is a bipartite graph with vertex set $V={[n]\choose k}\cup {[n]\choose k+1}$ and adjacencies $u\sim v$ if and only if either $u\subset v$ or $v\subset u$. In particular, note that if $n=2k-1$, then $J(2k-1;k-1,k)$ is  the double odd graph $2O_{k}$.
For other values of $n$, the doubled Johnson graph is not regular. More precisely, the degree of each vertex in ${[n]\choose k}$ is $n-k$, whereas the degree of each vertex in ${[n]\choose k+1}$ is $k+1$. Thus, the doubled Johnson graph $J(n;k,k+1)$ has
${n\choose k}+{n\choose k+1}={n+1\choose k+1}$ vertices and $(n-k){n\choose k}=(k+1){n\choose k+1}$ edges, and, as expected, it satisfies the symmetric property $J(n;k,k+1)\cong J(n;n-k-1,n-k)$.

Since the {\em doubled Johnson graphs} are  distance biregular bipartite graphs (see Delorme \cite{d94}), their adjacency and Laplacian spectra can be computed from the respective quotient matrices (see Dalf\'o and Fiol \cite{df20}).
For instance, the Laplacian eigenvalues of $J(n;k,k+1)$ are the following:
\begin{itemize}
\item
If $n$ is even and $k\le n/2-1$ (the case $k=n/2$ corresponds to the already mentioned double odd graphs),
$$
\lambda_j=j\quad  \mbox{for } j=0,\ldots,k-1,\qquad \mbox{and}\qquad \lambda_{k+j}=n-k+j\quad \mbox{for } j=1,\ldots,k;
$$
\item
If $n$ is odd,  and $k\le (n-1)/2$,
$$
\lambda_j=j\quad \mbox{for } j=0,\ldots,k,\qquad \mbox{and}\qquad \lambda_{k+j}=n-k+j\quad \mbox{for } j=1,\ldots,k.
$$
\end{itemize}

By using the same mapping that shows the isomorphism between $F_k(S_{2k})$  and $2O_{k}\cong J(2k-1;k-1,k)$, it is readily seen
that  $F_{k}(S_{n-1})\cong J(n;k-1,k)$.

\section{A graph, its complement, and their token graphs}
\label{sec:-6}

Let us consider a graph $G$ and its complement $\overline{G}$, with respective Laplacian matrices $\L$ and $\overline{\L}$.
We already know that the eigenvalues of $G$ are closely related to the eigenvalues of $\overline{G}$, since $\L+\overline{\L}=n\I-\J$.
Hence, the same relationship holds for the algebraic connectivity, see Fiedler \cite{fi73}.

Observe that the $k$-token graph of $\overline{G}$ is the complement of the $k$-token graph of $G$ with respect to the Johnson graph $J(n,k)$ (the $k$-token graph of $K_n$), see Carballosa, Fabila-Monroy, Lea\~nos, and Rivera \cite[Prop. 3]{cflr17}.
Then, it is natural to ask whether a similar relationship holds between the Laplacian spectrum of the $k$-token graph of $G$ and the Laplacian spectrum of the $k$-token graph of $\overline{G}=K_n-G$.
In this section, we prove that, indeed, this is the case.


Our result is a consequence of the following property.

\begin{lemma}
\label{lem:com}
Given a graph $G$ and its complement $\overline{G}$, the Laplacian matrices $\L=\L(F_k(G))$ and $\overline{\L}=\L(F_k(\overline{G}))$ of their $k$-token graphs commute:
$
\L\overline{\L}=\overline{\L}\L.
$
\end{lemma}

\begin{proof}
 We want to prove that $(\L \overline{\L})_{AB}=(\overline{\L}\L)_{AB}$ for every pair of vertices $A,B\in {[n]\choose k}$ of $F_k(G)$ and $F_k(\overline{G})$, respectively. To this end, we consider the different possible values of $|A\cap B|$.
First, note that $(\L \overline{\L})_{AB}=(\overline{\L}\L)_{AB}=0$ when $|A\cap B|<k-1$. Moreover, if $|A\cap B|=k$, that is $A=B$, we have $(\L \overline{\L})_{AA}=(\overline{\L} \L)_{AA}=\deg_{F_k(G)}(A)\cdot \deg_{F_k(\overline{G})}(A)$. Thus, we only need to consider the case $|A\cap B|=k-1$, where we can assume that $A=A'\cup \{a\}$ and $B=A'\cup \{b\}$, where $|A'|=k-1$ and $a,b\in[n]$ for $a\neq b$.
Moreover, without loss of generality, we can  assume that $a\sim b$ in $G$, which implies that $a\not\sim b$ in $\overline{G}$. (If not, interchange the roles of $A$ and $B$.)
Then, the different terms of the sum
$$
(\L \overline{\L})_{AB}= \sum_{X\in {[n]\choose k}} (\L)_{AX}(\overline{\L})_{XB}
$$
can be seen as walks of length two, where the first step is done in $G$ and the second step is done in $\overline{G}$,
yielding the following possibilities:
\begin{enumerate}
\item[$(i)$]
First `add $b$', and then `delete $a$':
$$
A=A'\cup \{a\} \quad \rightarrow \quad X=(A'\setminus \{c\})\cup\{b\} \quad \rightarrow \quad (X\setminus \{a\}) \cup \{c\}=B,
$$
where $c\in A'$, $c\sim b$ in $G$; and $c\sim a$ in $\overline{G}$. Thus, for each $c\in N_G(b)\cap N_{\overline{G}}(a)\cap A'$ we get a term $(-1)(-1)=1$.
\item[$(ii)$]
First `delete $a$', and then  `add $b$':
$$
A=A'\cup \{a\} \quad \rightarrow \quad X=(A\setminus \{a\})\cup\{c\}
\quad \rightarrow \quad (X\setminus \{c\}) \cup \{b\}=B,
$$
where $c\in (A\cup \overline{B}$, $c\sim a$ in $G$; and $c\sim b$ in $\overline{G}$. Thus, for each $c\in N_G(a)\cap N_{\overline{G}}(b)\cap \overline{A\cup B}$, we get a term $(-1)(-1)=1$.
\item[$(iii)$]
First `change $a$ by $b$', and then  `keep $b$':
$$
A=A'\cup \{a\} \quad \rightarrow \quad X=(A\setminus \{a\})\cup\{b\}=B \quad \rightarrow \quad B,
$$
where $a\sim b$ in $G$. This corresponds to the case when $X=B$ and $A\sim B$ in $F_k(G)$. Then, since the second step corresponds to the diagonal entry of $(\overline{\L})_{BB}$, this gives the term $(-1)\cdot\deg_{F_k(\overline{G})} B$.
\item[$(iv)$]
The other way around (first  `keep $a$', and then `change $a$ by $b$') corresponds to the case when $X=A$, but $A\not\sim B$ in $F_k(\overline{G})$. Then, this  gives zero since $a\not\sim b$ in $\overline{G}$.
\end{enumerate}
Summarizing, since all the other possibilities give a zero term, we conclude that, when $|A\cap B|=k-1$,
\begin{equation}
\label{LLC(AB)}
(\L\overline{\L})_{AB}=|N_G(b)\cap N_{\overline{G}}(a)\cap A'|+|N_G(a)\cap N_{\overline{G}}(b)\cap \overline{A\cup B}|-\deg_{F_k(\overline{G})} (B).
\end{equation}
With respect to  $(\overline{\L}\L)_{AB}$, note that, since the involved matrices are symmetric, we can compute
$(\overline{\L}\L)_{BA}=\sum_{X\in{[n]\choose k}}(\overline{\L})_{BX}(\L)_{XA}$. Reasoning as before, the walks $(i)$ and $(ii)$ (from $A$ to $B$) become the walks for $B$ to $A$ if we interchange these two vertices, the vertices $a$ and $b$, and the graphs $G$ and $\overline{G}$.
For instance, $(i)$ and $(iv)$ lead to:
\begin{itemize}
\item[$(i')$]
First `add $a$', and then  `delete $b$':
$$
B=A'\cup \{b\} \quad \rightarrow \quad X=(A'\setminus \{c\})\cup\{a\} \quad \rightarrow \quad (X\setminus \{b\}) \cup \{c\}=A,
$$
where $c\in A'$, $c\sim a$ in $\overline{G}$; and $c\sim b$ in $G$. Thus, for each $c\in N_G(b)\cap N_{\overline{G}}(a)\cap A'$, we get again a term $(-1)(-1)=1$.
\item[$(iii')$]
First  `keep $b$', and then  `change $b$ by $a$':
$$
B=A'\cup \{a\} \quad \rightarrow \quad X=B \quad \rightarrow \quad (B\setminus \{b\})\cup\{a\}=A,
$$
where $a\sim b$ in $G$. The first step corresponds again to the diagonal entry of  $(\overline{\L})_{BB}$, which gives the term $(-1)\cdot\deg_{F_k(\overline{G})} (B)$ again.
\end{itemize}
Consequently, considering all the above cases, we have that $(\L\overline{\L})_{AB}=(\overline{\L}\L)_{AB}$ for any $A$ and $B$, as claimed.
\end{proof}

\begin{theorem}
\label{theo:pairing}
Let $G=(V,E)$ be a graph on $n=|V|$ vertices, and let $\overline{G}$ be its complement. For a given $k$, with $1\leq k<n-1$, let us consider the token graphs $F_k(G)$ and  $F_k(\overline{G})$. Then, the Laplacian spectrum of $F_k(\overline{G})$ is the complement of the Laplacian spectrum of $F_k(G)$ with respect to the Laplacian spectrum of the Johnson graph $J(n,k)=F_k(K_n)$.
That is, every eigenvalue  $\lambda_J$ of $J(n,k)$ is the sum of one eigenvalue $\lambda_{F_k(G)}$ of $F_k(G)$ and one eigenvalue $\lambda_{F_k(\overline{G})}$ of $F_k(\overline{G})$, where each $\lambda_{F_k({G})}$ and each $\lambda_{F_k(\overline{G})}$ is used once:
$$
\label{spFk(G)-sp(Fk(noG))}
\lambda_{F_k({G})}+\lambda_{F_k(\overline{G})}=\lambda_J.
$$
\end{theorem}

\begin{proof}
Let $\L$ and $\overline{\L}$ be the Laplacian matrices of $F_k(G)$ and  $F_k(\overline{G})$, respectively. Let $\L_J$ be the Laplacian matrix of $J(n,k)$. Since $\L+\overline{\L}=\L_J$, Lemma \ref{lem:com} implies that $\L_J$ commute with both $\L$ and $\overline{\L}$. Indeed,
$\L_J\L=(\L+\overline{\L})\L=\L(\L+\overline{\L})=\L\L_J$. This implies that the matrices $\L$, $\overline{\L}$, and $\L_J$ diagonalize simultaneously, that is, the vertex space $V\cong \Re^n$ has a basis consisting of
common eigenvectors of all $\L$, $\overline{\L}$, and $\L_J$  (see, for instance, Brouwer and Haemers \cite[Prop. 2.1.1]{bh10}).
Then, for each common eigenvector $\v$,
$$
\L\v+\overline{\L}\v=\L_J\v\qquad \Rightarrow \qquad \lambda_{F_k({G})}+\lambda_{F_k(\overline{G})}=\lambda_J.
$$
Since there are $n$ independent (common) eigenvectors, the result follows.
\end{proof}

Let us show an example.
%
\begin{example}
Consider the graphs $G$, its 2-token graph $F_2(G)$, the complement $\overline{G}$ of $G$, and its 2-token graph  $F_2(\overline{G})$ in Figure \ref{fig:C_3-amb-aresta+convers+tokens}. Their Laplacian matrices are the following:

	\begin{align*}
	&\L(G)=
	\left(\begin{array}{rrrr}
		1&-1&0&0\\
		-1&3&-1&-1\\
		0&-1&2&-1\\
		0&-1&-1&2
	\end{array}\right), \qquad
	\L(F_2(G))=
	\left(\begin{array}{rrrrrr}
		2&-1&-1&0&0&0\\
		-1&3&-1&-1&0&0\\
		-1&-1&3&0&-1&0\\
		0&-1&0&3&-1&-1\\
		0&0&-1&-1&3&-1\\
		0&0&0&-1&-1&2
	\end{array}\right),\\
&\L(\overline{G})=
\left(\begin{array}{rrrr}
    2&0&-1&-1\\
	0&0&0&0\\
	-1&0&1&0\\
	-1&0&0&1
\end{array}\right), \qquad\ \
\L(F_2(\overline{G}))=
\left(\begin{array}{rrrrrr}
	2&0 &0 &-1&-1&0\\
	0&1 &0 & 0& 0&-1\\
	0&0 &1 & 0& 0&-1\\
   -1&0 &0 & 1& 0&0\\
   -1&0 &0 & 0& 1&0\\
	0&-1&-1& 0& 0&2
\end{array}\right).
\end{align*}
From Lemma \ref{lem:com}, one can check that $\L(F_2(G))$ and $\L(F_2(\overline{G}))$ commute. Of course,  the same holds for $\L(G)$ and $\L(\overline{G})$ since $\L(G)+\L(\overline{G})=4\I-\J$.
From this, the eigenvalues of $K_4$ are obtained by adding one of the eigenvalues of $G$ plus one of $\overline{G}$. Now, note that the same happens for the eigenvalues of the Johnson graph $J(4,2)=F_2(G)\cup F_2(\overline{G})$. Indeed, as shown in the figure, the eigenvalues of $J(4,2)$ can be obtained by adding one eigenvalue of  $F_2(G)$ to one eigenvalue of $F_2(\overline{G})$, in such a way that each eigenvalue of both graphs is used exactly once.
\end{example}

\begin{figure}[t]
	\begin{center}
		\includegraphics[width=\textwidth]{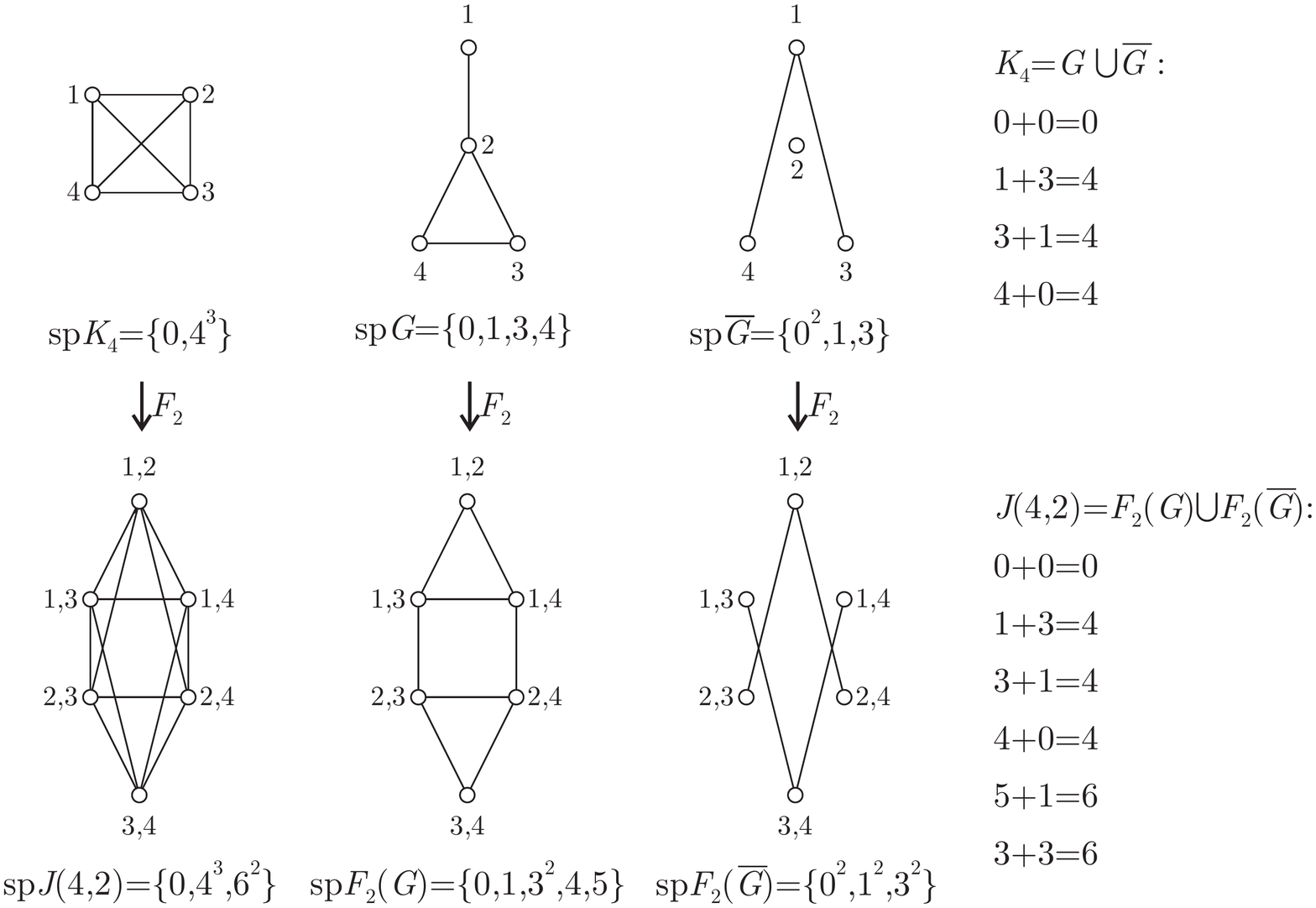}
	\end{center}
	\vskip-1.5cm
	\caption{The graphs $K_4$, $G$, its complement $\overline{G}$ and their 2-token graphs.}
	\label{fig:C_3-amb-aresta+convers+tokens}
\end{figure}

Since all the eigenvalues of the Johnson graphs are integers, Theorem \ref{theo:pairing} leads to the following consequence.
\begin{corollary}
\label{coro:comp}
Let $G$ be a graph such that its complement $\overline{G}$ has $c$ connected components. Then, for $1\leq k \leq n-1$, the $k$-token graph
	$F_k(G)$ has at least $c$ integer eigenvalues. If each of the $c$ components of $\overline{G}$ has at least $k$ vertices,
	then $F_k(G)$ has at least ${{c+k-1}\choose{k}}$ integer eigenvalues.
\end{corollary}
\begin{proof}
The multiplicity of the zero eigenvalue of $F_k(\overline{G})$ equals the number of its connected components. If $\overline{G}$ has $c$  components, say $H_1,\ldots,H_c$, then  $F_k(\overline{G})$ has  at least $c$ components. In fact, each component is some Cartesian product of the form $F_{k_1}(H_1)\times\cdots\times F_{k_c}(H_c)$ with $\sum_{i=0}^c k_i\le k$ and $k_i\le |V(H_i)|$ for $i=1,\ldots,c$. Thus, the number of
	components of $F_k(\overline{G})$ equals the number of ways of distributing $k$ indistinguishable tokens among the $c$ components of $\overline{G}$. If each component of $\overline{G}$ has at least $k$ vertices, then this number is ${{c+k-1}\choose{k}}$.  Then, Theorem~\ref{theo:pairing} completes the proof.
\end{proof}


\section{The algebraic connectivity}
\label{sec:-7}

In this last section, we study the algebraic connectivity of the token graphs. In the proof of our main result, we use the following concepts and lemma.
\\
Let $G$ be a graph with $k$-token graph $F_k(G)$.
For a vertex $a\in V(G)$, let $S_a:=\{A\in V(F_k(G)):A\ni a\}$ and $S'_a:=\{B\in V( F_k(G)): B\not\ni a\}$.
Let $H_a$ and $H'_a$ be the subgraphs of $F_k(G)$ induced by $S_a$ and $S'_a$, respectively.
Note that $H_a\cong F_{k-1}(G\setminus \{a\})$ and $H'_a\cong F_k(G\setminus \{a\})$.
\begin{lemma}
\label{lem:eigenvectors}
Given a vertex $a\in G$ and an eigenvector $\v$ of $F_k(G)$ such that $\B^{\top}\v=\vec0$, let
\[
\w_a:=\restr{\v}{S_a} \text{ and \quad } \w'_a:=\restr{\v}{S'_a}.
\]
Then, $\w_a$ and $\w'_a$ are embeddings of $H_a$ and $H'_a$, respectively.
\end{lemma}
\begin{proof}
Assume that the matrix $\B^{\top}$ has the first row indexed by   $a\in V(G)$. Then, we have
$$
\vec0=
\B^{\top}\v=
\left(
\begin{array}{c|c}
\1^{\top} & \vec0^{\top}\\
\hline
\B_1 & \B_2
\end{array}
\right)
\left(
\begin{array}{c}
\w_a \\
\hline
\w'_a
\end{array}
\right)=
\left(
\begin{array}{c}
\1^{\top}\w_a \\
\hline
\B_1\w_a+\B_2\w'_a
\end{array}
\right),
$$
where $\1^{\top}$ is a row ${n-1\choose k-1}$-vector, $\vec0$ is a row  ${n-1\choose k}$-vector, $\B_1=\B(n-1,k-1)^{\top}$, and  $\B_2=\B(n-1,k)^{\top}$.
Then, $\1^{\top}\w_a=0$, so that $\w_a$ is an embedding of $H_a$. Furthermore, since $\v$ is an embedding of $G$, we have $\1^{\top}\v=\1^{\top}\w_a+\1^{\top}\w'_a=0$ (with the appropriate dimensions of the all-1 vectors). Hence, it must be  $\1^{\top}\w'_a =0$, and $\w'_a$ is an embedding of $H'_a$.
\end{proof}

In the following result, we give the algebraic connectivity of some infinite families of graphs.
\begin{theorem}
\label{theo:alg-connec}
For each  of the following classes of graphs, the algebraic connectivity of a token graph $F_k(G)$ equals the algebraic connectivity of $G$.
\begin{itemize}
\item[$(i)$]
Let $G=K_n$ be the complete graph on $n$ vertices. Then,
$\alpha(F_k(G))=\alpha(G)=n$ for every $n$ and $k=1,\ldots,n-1$.
\item[$(ii)$]
Let $G=S_n$ be the star graph on $n$ vertices. Then,
$\alpha(F_k(G))=\alpha(G)=1$ for every $n$ and $k=1,\ldots,n-1$.
\item[$(iii)$]
Let $G=P_n$ be the path graph on $n$ vertices. Then, $\alpha(F_k(G))=\alpha(G)\linebreak =2(1-\cos(\pi/n))$ for every $n$ and  $k=1,\ldots, n-1$.
\item[$(iv)$]
Let $G= K_{n_1,n_2}$ be the complete bipartite graph on $n=n_1+n_2$ vertices, with $n_1\le n_2$. Then, $\alpha(F_k(G))=\alpha(G)=n_1$ for every $n_1,n_2$ and $k=1,\ldots,n-1$.
\end{itemize}
\end{theorem}
\begin{proof}
		$(i)$ and $(ii)$ follows from the results of Section \ref{sec:-5} about the spectra of the double odd graphs and doubled Johnson graphs.
		To prove $(iii)$, let $V(P_n)=[n]$ with $i\sim i+1$ for $i=1,\ldots, n-1$. The result is readily checked for $n\le 3$. For the other cases,
		we proceed by induction on $n(\ge 4)$ and $k$. For $k=1$, the claim holds since $F_1(P_n)\simeq P_n$. For $n=4$,
		we can assume that $k=1$ or $k=2$,  and in both cases, the claim holds. So, the induction starts.
		Suppose $n>4$ and $k>1$. To our aim, by Corollary \ref{coro:LkL1}$(ii)$, it suffices to show that if $\v$ is an eigenvector of $F_k:=F_k(P_n)$,
		with $\B^{\top}\v=\vec0$, then $\lambda(\v)\geq \alpha(P_n)$.
		Without loss of generality, we may assume that $\v^{\top}\v=1$.
		As defined before, let $S_n:=\{A\in V(F_k):A \ni n\}$ and $S'_n:=\{B\in V(F_k):B\not\ni n\}$.
		Let $H_n$ and $H'_n$ be the subgraph of $F_k$ induced by $S_n$ and $S'_n$, respectively. We have $H_n\cong F_{k-1}(P_{n-1})$ and
		$H'_n\cong F_k(P_{n-1})$. Let $\w_n:=\restr{\v}{S_n}$ and $\w'_n:=\restr{\v}{S'_n}$, by Lemma~\ref{lem:eigenvectors},
we know that $\w_n$ and $\w'_n$ are embeddings of $H_n$ and $H'_n$, respectively. By the induction hypothesis, we have
\[
		\lambda(\w_n)=\frac{\sum\limits_{(A,B)\in E(H_n)}(\w_n(A)-\w_n(B))^2}{\sum\limits_{A\in V(H_n)}\w_n(A)^2}\geq \alpha(P_{n-1}),
\]
		and
		\[
		\lambda(\w'_n)=\frac{\sum\limits_{(A,B)\in E(H'_n)}(\w'_n(A)-\w'_n(B))^2}{\sum\limits_{A\in V(H'_n)}\w'_n(A)^2}\geq \alpha(P_{n-1}).
		\]
		Since $V(H_n)\cup V(H'_n)=V(F_k)$ and $\v^{\top}\v=1$, we have
\begin{align}
\lambda(\v)&=\sum\limits_{(A,B)\in E(F_k)}(\v(A)-\v(B))^2  \nonumber\\
 & \geq \sum\limits_{(A,B)\in E(H_n)}(\w_n(A)-\w_n(B))^2 + \sum\limits_{(A,B)\in E(H'_n)}(\w'_n(A)-\w'_n(B))^2 \nonumber\\
	 & \geq \alpha(P_{n-1})\Big[\sum\limits_{A\in V(H_n)}\w_n(A)^2 + \sum\limits_{B\in V(H'_n)}\w'_n(B)^2\Big] \nonumber \\
	& = \alpha(P_{n-1})\Big[ \sum\limits_{A\in V(H_n)}\v(A)^2 + \sum\limits_{B\in V(H'_n)}\v(B)^2\Big] \nonumber \\
	& = \alpha(P_{n-1}) > \alpha(P_n), \label{eq:paths-1}
\end{align}
		where (\ref{eq:paths-1}) follows from the fact that $\alpha(P_n)=2(1-\cos(\pi/n))$   (see Anderson and Morley \cite{am85} ). Furthermore, since $\lambda(\v)>\alpha(P_n)$, we get that $\alpha(P_n)$ is an eigenvalue of
		both $P_n$ and $F_k(P_n)$ with the same multiplicity. \\
		
Regarding $(iv)$, it is known that the graph $G=K_{n_1,n_2}$, with $n_1\le n_2$,
has algebraic connectivity $\alpha(G)=n_1$. To prove that $\alpha(F_k(G))=n_1$, we first consider the token graph of its complement $\overline{G}=K_{n_1}\cup K_{n_2}$. From Section \ref{sec:-5}, we know that $F_k(K_n)=J(n,k)$  with first different eigenvalues $\theta_1=0,\theta_2=n,\theta_3=2(n-1),\ldots$
Hence,
$$
F_k(\overline{G})=\bigcup_{h=0}^{\min\{k,n_1\}} [J(n_1,h)\times J(n_2,k-h)],
$$
(see the proof of Corollary \ref{coro:comp}). 		
Now, recall that the eigenvalues of a Cartesian product of two graphs $G_1,G_2$ are obtained as all possible sums of one eigenvalue of $G_1$ with one eigenvalues of $G_2$. Then,  the first different eigenvalues of $F_k(\overline{G})$ are $\sigma_1=0$, $\sigma_2=n_1$, $\sigma_3=n_2$, $\sigma_4=n_1+n_2=n$, \ldots 
Moreover, by Theorem \ref{theo:pairing}, the eigenvalues $\lambda_1\le \lambda_2\le \cdots \le \lambda_{{n\choose k}}$ of $F_k(G)$ are of the form
$
\lambda_i=\theta_r-\sigma_s
$
for some $i,r,$ and $s$. With the obtained values, we can represent this as follows:
$$
\{\lambda_1,\lambda_2,\ldots\}=\{0,n,\ldots,n,2(n-1),\ldots\}-\{0,\ldots,0,n_1,\ldots,n_1,n_2,\ldots,n_2,n,\ldots\}.
$$
This help us to understand which are the two first eigenvalues of $F_k(G)$. Indeed, $\lambda_1=0-0=0$ and $\lambda_2=\alpha(F_k(G))>0$ since  $F_k(G)$ is connected, with the following  possible cases:\\
$(a)$ $ \lambda_2=n-0=n$ is not possible because, from Theorem \ref{th:(1,k)}, $\alpha(F_k(G))$ can not be greater than $\alpha(G)$.\\
$(b)$ By the same reason, it can not be  $\lambda_2=n-n_1=n_2$, unless $n_1=n_2$ in which case $\alpha(F_k( K_{n_1,n_2}))=\alpha(K_{n_1,n_2})=n_1$.
Thus, we conclude that $\lambda_2=n-n_2=n_1=\alpha(G)$, as claimed, since all the other possible cases are trivially impossible.
\end{proof}
In fact, this last case $(iv)$ can also be proved by using the method in the proof of $(iii)$. Besides, we think that both methods could be used for other families of graphs.

Together with all the above results,  computer exploration showed that $\alpha(F_2(G))=\alpha(G)$ for all graphs with at most $8$ vertices. These results lead us to formulate the following conjecture.
\begin{conjecture}
	The algebraic connectivity of the token graph of a graph  equals the algebraic connectivity of the original graph. That is,
$$
\alpha(F_k(G))=\alpha(G)\quad \mbox{for every $k=1,\ldots,|V|-1$}.
$$
\end{conjecture}
Observe that this conjecture only needs to be proved for the case $k=\lfloor n/2 \rfloor$ because of Corollary \ref{coro2}.
Note also that the conjecture trivially holds when the graph $G$ is disconnected, since $F_k(G)$ is also disconnected and, hence, $\alpha(G)=\alpha(F_k(G))=0$.
Moreover, from Theorem \ref{theo:alg-connec}, the conjecture also holds for those graphs whose token graphs are regular. Namely, $K_n$, $S_n$ (with even $n$ and $k=n/2$), and their complements.


\end{document}